\documentclass[a4paper,11pt]{article}
\usepackage[hyperfootnotes=true]{hyperref}
\usepackage[top=1.25in, bottom=1.25in, left=1.25in, right=1.25in]{geometry}
\usepackage{amsfonts}
\usepackage{enumitem}
\usepackage{amsmath}
\usepackage{amsthm}
\usepackage{amssymb}
\usepackage{url}
\usepackage{bm}
\usepackage{bbm}
\usepackage{booktabs}
\usepackage{xcolor}
\usepackage{comment}
\usepackage{adjustbox}
\usepackage{mathtools}
\usepackage{siunitx}
\usepackage{setspace,lineno}
\usepackage{natbib}
\usepackage{float}
\usepackage{algorithm}
\usepackage{algpseudocode}
\usepackage{xcolor,algorithmicx,algorithm}
\usepackage{graphicx}
\usepackage{subcaption}
\usepackage{tikz-cd}
\usepackage{pgfplots}
\pgfplotsset{compat=1.17}
\usepackage{pgfplotstable}
 \usepgfplotslibrary{fillbetween}
\usepackage{graphicx}
\hypersetup{colorlinks=true,raiselinks=true,breaklinks=true,citecolor=blue}

\graphicspath{ {./pix} }
\setlist[enumerate]{leftmargin=.5in}
\setlist[itemize]{leftmargin=.5in}
\DeclareMathOperator{\spn}{span}

\DeclareMathOperator{\trace}{trace}

\DeclareMathOperator*{\argmin}{arg\,min \, }

\newcommand{\T}{{\top}}

\newcommand{\E}{{\mathbb E}}

\newcommand{\Q}{{\mathbb Q}}

\newcommand{\R}{{\mathbb R}}

\newcommand{\Fcal}{{\mathcal F}}
\newcommand{\Gcal}{{\mathcal G}}
\newcommand{\Hcal}{{\mathcal H}}

\newcommand{\Ocal}{{\mathcal O}}

\newcommand{\Rcal}{{\mathcal R}}

\newcommand{\Vcal}{{\mathcal V}}

\newcommand{\Xcal}{{\mathcal X}}
\newcommand{\Ycal}{{\mathcal Y}}
\newcommand{\Zcal}{{\mathcal Z}}
\newcommand{\bs}[1]{{\bm #1}}

\newcommand{\dd}{\operatorname{d}\!}

\newcommand{\isdef}{\mathrel{\mathrel{\mathop:}=}}

\newcommand{\StateIndent}{\hspace{\algorithmicindent}}

\newtheorem{proposition}{Proposition}[section]

\newtheorem{theorem}[proposition]{Theorem}

\newtheorem{remark}[proposition]{Remark}

\definecolor{darkred}{rgb}{0.9,0,0}

\title{Low-rank kernel methods for American option pricing}
\date{}
\author{Michael Multerer\footnote{
		Universit\`a della Svizzera italiana, 6900 Lugano, 
  Switzerland (michael.multerer@usi.ch).}
	\and
	Paul Schneider\footnote{
		Universit\`a della Svizzera italiana, 6900 Lugano, Switzerland, 
    and Swiss Finance Institute, 8006 Z\"urich, Switzerland
        (paul.schneider@usi.ch).}
	\and
	Chiara Segala\footnote{
		Universit\`a della Svizzera italiana, 6900 Lugano, 
  Switzerland (chiara.segala@usi.ch).}
        }

\begin{document} 
\maketitle

\begin{abstract}
We propose a scalable and theoretically grounded low-rank conditional
expectation model for recursive Monte Carlo optimal stopping problems, in
particular American option pricing.
Our method reformulates the estimation of continuation values as a learning
problem in a reproducing kernel Hilbert space, in which the conditional
expectation is represented as a linear operator acting on future payoffs. This
perspective yields an offline-online decomposition: the operator is learned once
from simulated data and subsequently reused across all exercise dates,
eliminating the need to recompute regression models at each step of the backward
recursion.
We establish convergence guarantees and derive bounds quantifying the
approximation errors across exercise dates. Numerical experiments demonstrate
the speed and accuracy  of the proposed approach relative to extant methods.
\end{abstract}



\section{Introduction}
The valuation and optimal exercise of American-style derivatives is one of the
most enduring problems in mathematical finance. Unlike their European
counterparts, American options grant the holder the right to exercise at any
time prior to maturity, and their fair price is therefore the value of an
optimal stopping problem under a risk-neutral measure \citep{karatzas1998}.
Closed-form solutions are available only in a handful of special cases, while
classical lattice and PDE-based techniques
\citep{brennanschwartz77,coxrossrubinstein79,baroneadesiwhaley87} scale poorly
with the dimension of the state vector. As pricing problems in modern markets
routinely involve baskets of assets, stochastic volatility, and stochastic
interest rates, the development of accurate and computationally tractable
methods for high-dimensional optimal stopping remains an active and pressing
research topic.

The dominant paradigm for high-dimensional American option pricing to date is
the regression-based least-squares Monte Carlo (LSM) approach, introduced 
by \citet{carriere96}, \citet{tsitsiklisvanroy01} and 
\citet{LongstaffSchwartz01}. The method approximates the conditional
continuation value at each exercise date by projecting future discounted cash
flows onto a finite linear span of basis functions of the current state, and
uses the resulting estimates inside a backward dynamic programming recursion.
Its conceptual simplicity, ease of implementation, and compatibility with
arbitrary path generators have made LSM the standard benchmark in both industry
and academic work \citep{glasserman04}. Complementary dual formulations were
developed by \citet{rogers02} and \citet{haugh-kogan04}.

The theoretical analysis of LSM has progressed in parallel with its widespread
adoption. Early consistency results were established by
\citet{clementlambertonprotter02} for fixed finite-dimensional spans and by
\citet{stentoft04} for sieve-type bases. \citet{Egloff2005} embedded LSM in the
framework of empirical risk minimization for optimal stopping, deriving
finite-sample bounds under convexity and closure assumptions on the
approximating class. A unifying convergence theory for LSM-type schemes was
developed by Zanger in a sequence of works
\citep{Zanger2009123,Zanger2013503,Zanger2018447,Zanger2020923}, 
covering settings ranging from bounded $L^2$-approximators of finite 
Vapnik-Chervonenkis dimension to nonlinear neural-network classes with 
possibly unbounded payoffs, and progressively relaxing assumptions on convexity, 
closure, and independence of sample paths.

A second line of recent research bypasses regression entirely and parameterizes
either the stopping rule, the value function, or an associated PDE solution by
deep neural networks. \citet{kohlerkrzyzaktodorovic10} pioneered the use of
feed-forward networks within the LSM regression step. The deep optimal stopping
framework of \citet{becker2019deep} parameterizes the stopping decision at each
date directly as a neural network and learns it by stochastic gradient ascent on
the expected payoff, while \citet{becker2021solving} extends this idea to
genuinely high-dimensional optimal stopping problems with rigorous error
analysis. The closely related contribution of \citet{becker2020pricing} develops
a unified deep-learning methodology for both pricing and hedging American-style
options and demonstrates its accuracy on benchmarks involving up to several
hundred underlyings. Complementary approaches solve the associated free-boundary
or variational problem directly: \citet{hanjentzene18} reformulate
high-dimensional PDEs as backward stochastic differential equations approximated
by neural networks. \citet{sirignanospiliopoulos18} introduce the deep Galerkin
method, \citet{hurephamwarin20} propose deep backward schemes for nonlinear
PDEs, including those associated with optimal stopping,
and \citet{lapeyrelelong21} study neural-network regression in the Bermudan
setting. While these methods deliver impressive performance in high dimensions,
they typically require costly retraining of network parameters at each exercise
date or for each new contract, and their statistical properties are still less
well understood than those of regression-based estimators.

The present work develops an alternative approach that combines the theoretical
transparency of regression-based LSM with the data-driven flexibility of modern
machine learning. We replace the explicit and oftentimes ad-hoc choice of basis
functions in LSM by a nonparametric estimator in a reproducing kernel Hilbert
space (RKHS) and reformulate the estimation of continuation values through the
\emph{conditional mean embedding} (CME) operator 
\citep{songetal09,gruenewaelderetal12,smola2007hilbert,
klebanovschustersullivan20,NEURIPS2020_f340f1b1}. This perspective reveals that
the conditional expectation entering the dynamic programming recursion is a
single linear operator on the RKHS, which can be learned \emph{once} and
\emph{offline} from simulated paths and reused across all exercise dates and
strikes considered. We thereby eliminate the per-step regression that dominates
the computational cost of LSM and similar deep-learning schemes, and obtain a
clean offline-online decomposition.

Naive kernel methods, however, suffer from the well-known $\mathcal{O}(n^3)$
scaling in the number of simulated paths $n$. To restore practicality we exploit
the spectral decay of the relevant kernel matrices and introduce a low-rank
approximation based on the pivoted Cholesky decomposition \citep{HPS12}, in the
spirit of Nystr\"om-type approximations \citep{Seeger}. This yields a
closed-form estimator in a reduced subspace whose cost scales linearly in $n$
and whose accuracy can be controlled through the truncation tolerance.
Our contributions can be summarized as follows.
\begin{enumerate}
    \item We reformulate the continuation value estimation step in American
      option pricing as the learning of a conditional mean embedding in a
      tensor-product RKHS,  leading to an offline-online algorithm that learns
      the conditional expectation  operator once and reuses it across all
      exercise dates.
    \item We introduce a low-rank approximation of the resulting kernel system
      based on pivoted Cholesky factorization, yielding a closed-form estimator
      in a reduced subspace at substantially lower computational cost than the
      dense kernel solution.
    \item We establish convergence of the proposed scheme and derive explicit
      bounds on the approximation errors that propagate through the backward
      recursion. The results extend the existing convergence theory of LSM-type
      algorithms to nonparametric kernel estimators with low-rank truncation.
    \item We provide numerical experiments on standard benchmarks demonstrating
      that the method matches or outperforms the Longstaff-Schwartz benchmark in
      accuracy while offering a substantial speed-up.
\end{enumerate}

The remainder of the paper is organized as follows. Section~\ref{sec:american}
introduces the problem of American option pricing and reviews simulation-based
regression. Section~\ref{sec:cme} presents the conditional mean embedding
framework and develops the low-rank kernel estimator based on pivoted Cholesky
decomposition. Section~\ref{sec:conv_rate} establishes the convergence and error
analysis of the proposed scheme, and Section~\ref{sec:offline_cme} extends the
analysis to the backward recursion. Section~\ref{sec:num_test} reports numerical
results on standard benchmarks, and Section~\ref{sec:conclusion} concludes.

\section{American option pricing}\label{sec:american}

One of the central challenges in option pricing theory is the valuation and
optimal exercise of derivatives with American-style exercise features. Such
instruments appear across major financial markets and despite significant
progress in quantitative finance, valuing and optimally exercising American
options remains difficult, especially when multiple factors influence the
option's value. At any potential exercise date, the holder of an American option
must compare the immediate exercise payoff with the expected continuation value.
The option should be exercised when the immediate payoff exceeds the expected
value of keeping the option alive. Consequently, the optimal exercise strategy
depends on accurately estimating the conditional expectation of the continuation
payoff. Our approach addresses this by estimating the conditional expectation
through a kernel-based conditional mean embedding approach.

\subsection{Optimal stopping and conditional expectation framework}
Let $(\Omega, \mathcal{F}, \mathbb{Q})$ be a complete probability space and
consider a finite time horizon $[0, T]$. We assume the existence of a filtration
$\{\mathcal{F}_t\}_{t \in [0, T]}$ generated by the underlying price process
$\{S_t\}_{t\in[0,T]}$, i.e., $\mathcal{F}_t=\sigma(\{S_u : 0\leq u\leq t\})$,
with $\mathcal{F}_T = \mathcal{F}$. Following the no-arbitrage framework, the
measure $\mathbb{Q}$ is taken to be a risk-neutral pricing measure, 
under which discounted price processes are martingales and expectations 
correspond to risk-neutral pricing, i.e.,
\[
\mathbb{E}\left[
  e^{-r t} S_t
  \,\middle|\,
  \mathcal{F}_s
\right]
=
e^{-r s} S_s,
\quad s\in[0,t],\ t\leq T,
\]
for some constant risk-free interest rate $r \in \mathbb{R}$.
Within this setting, the value of an American option equals the maximum
expected discounted value of its cash flows, where the maximization is
taken over all stopping times adapted to the filtration, as formalized 
in the foundational work \citet{karatzas1998}.

To illustrate the main idea, we consider the case in which an American
put option can only be exercised at discrete times
$0 < t_1 < t_2 < \dots < t_{n_T} = T$, $n_T \in \mathbb{N}$. 
Let $\{\mathcal{F}_k\}_{k=0,\dots,n_T}$ be the associated filtration
generated by $\{S_k\}_{k=0,\dots,n_T}$. At the maturity date $T$, the option
is exercised, if it is in the money. At any earlier exercise time, 
the investor decides whether to exercise or to continue holding the option 
until the next decision date. The optimal policy is then to exercise as soon 
as the immediate payoff equals or exceeds the expected continuation value.
At a given decision time $t_k$, the payoff from immediate exercise is
fully determined by the current state of the underlying and is therefore
$\mathcal{F}_k$-measurable. In contrast, the continuation value depends
on uncertain future price movements and is defined under the
risk-neutral measure $\mathbb{Q}$ as the conditional expectation of the
discounted future option value.
In the present setting, the optimal value process satisfies the backward
recursion
\begin{align}
    \mathcal{V}_{n_T} &= P_{n_T}, \nonumber\\
    \mathcal{V}_k &= \max\big\{P_k,C(t_k, S_k)\big\}, 
    \qquad k = 0, \dots, n_T - 1, \label{eq:backward_recursion_intro}
\end{align}
where $P_k = e^{-rt_k}(K - S_k)^+$ denotes the immediate exercise payoff 
of an American put option at time $t_k$, discounted to $t_0$, for a given 
strike $K>0$. The \emph{continuation value} is defined as
\begin{equation} \label{eq:cont_value}
    C(t_k, S_k) = \mathbb{E}\!\left[\mathcal{V}_{k+1} \mid \mathcal{F}_k\right],
\end{equation}
that is, the expected 
optimal value at the next exercise date, conditional on the current information. 
Note that $\mathcal{V}_{k+1}$ already incorporates all future optimal stopping 
decisions, so that the continuation value accounts for the full remaining 
optionality of the contract.

In this work, we restrict our analysis to derivative payoffs belonging to the
space of square-integrable functions $L^2(\Omega, \mathcal{F}, \mathbb{Q})$.
This assumption ensures that all relevant expectations and conditional
expectations are well-defined and finite under the risk-neutral measure.
In practical terms, most financial payoffs satisfy this property, since their
values depend on asset prices or state variables that, although random, are
typically modeled with distributions possessing finite second moments.

\subsection{Modern approaches to option pricing}
Over the past decades, the valuation of American-style derivatives has evolved
from analytical and lattice-based techniques
\citep{brennanschwartz77,coxrossrubinstein79,baroneadesiwhaley87} to more
flexible simulation and data-driven methods
\citep{tilley93,carriere96,tsitsiklisvanroy01,LongstaffSchwartz01,glasserman04}.
These modern approaches aim to efficiently approximate the optimal stopping
policy that determines when it is most profitable to exercise the option,
especially in settings where closed-form solutions are not available.

Simulation-based approaches to option pricing typically rely on a backward
dynamic programming framework. A classical and widely adopted method in this
category is the Longstaff-Schwartz algorithm \citet{LongstaffSchwartz01}. This
approach estimates the continuation value through a regression-based
approximation of the conditional expectation of future discounted payoffs. The
procedure proceeds backward from maturity: a large number of sample paths of the
underlying process are simulated, and at each time step $t_k$, for all paths
where the option is in the money, the algorithm fits a regression model of the
discounted future cash flows on a set of basis functions
$\{\phi_j(S_k)\}_{j=1}^J$ of the current state variable $S_k$. Formally,
\[
\tilde{C}(t_k, S_k) = \sum_{j=1}^{J} \beta_{k,j} \, \phi_j(S_k),
\]
where the coefficients $\beta_{k,j}$ are obtained by minimizing the squared
error between the realized discounted payoffs and their regression estimates.
While the Longstaff-Schwartz algorithm is conceptually simple and
computationally efficient, its accuracy depends heavily on the choice of basis
functions. A poorly chosen function space can lead to either overfitting or
underfitting, and achieving stable results often requires a large number of
simulated paths. Moreover, the method provides only a parametric approximation
of the continuation value, which may not capture complex nonlinear dependencies
present in high-dimensional problems.

In recent years, advances in machine learning and deep learning have inspired
alternative methods for American option pricing
\citep{kohlerkrzyzaktodorovic10,LinAlmeida21,becker2019deep,becker2020pricing,
becker2021solving,hanjentzene18,sirignanospiliopoulos18,hurephamwarin20,
lapeyrelelong21}. These approaches typically recast the problem as the
approximation of conditional expectations or value functions within a flexible,
data-driven framework, often leading to improved performance in high-dimensional
settings. In particular, \citet{becker2019deep} and \citet{becker2020pricing}
parameterize stopping rules and hedging strategies directly by deep neural
networks, while \citet{hanjentzene18}, \citet{sirignanospiliopoulos18}
and \citet{hurephamwarin20} solve the underlying free-boundary or backward
stochastic differential equation problem in high dimensions.

In the following section, we introduce a kernel-based alternative for estimating
continuation values. This approach replaces the explicit specification of basis
functions with a data-driven representation in reproducing kernel Hilbert spaces,
and provides a flexible and theoretically grounded way to approximate 
conditional expectations.

\section{Conditional mean embedding}\label{sec:cme}

In this section, we introduce the concept of conditional mean embedding,
originally proposed by \citet{songetal09}. We begin with a few preliminary
definitions and notions to set the stage.

Consider two random variables, $X$ and $Y$, taking values in the separable and
complete metric spaces $\Xcal$ and $\Ycal$, respectively. These variables
follow a joint probability distribution $\mathbb{Q}$ on the product space
$\Zcal\isdef\Xcal\times\Ycal$, which is equipped with a product metric
$d_\Zcal$. A common choice is the squared product metric
$d_\Zcal^2 = d_\Xcal^2 + d_\Ycal^2$. This general setting ensures the existence
of conditional distributions \citep[Theorem 10.2]{dud_02} and guarantees the
weak convergence of the empirical distributions introduced below. In practical
applications, $\Xcal$ and $\Ycal$ are typically subsets of Euclidean spaces.

Given a sample $\{z_i = (x_i, y_i)\}_{i=1}^{n} \subset \Zcal$, we consider a
probability measure $\tilde{\mathbb{Q}}$ supported on the finite grid 
\[
    \Gcal \isdef \{(x_i, y_j) : i, j = 1, \dots, n\} \subset \Zcal,
\]
which serves as an approximation of $\mathbb{Q}$. The measure
$\tilde{\mathbb{Q}}$ allows for efficient computation of various functionals,
including conditional expectations such as
\begin{equation}\label{eq:cme}
    \tilde\E[ f(Y)|X=x],
\end{equation}
for a function $f$ defined on $\Ycal$. For notational simplicity, in what
follows we adopt this generic CME notation to represent conditional
expectations. The connection between this formulation and the notation
introduced in the option pricing framework is clarified in
Remark~\ref{rmk:cme_pricing}.

\begin{remark}[{Connection with option pricing notation}]\label{rmk:cme_pricing}
In the context of option pricing, the CME formulation \eqref{eq:cme} can be
directly related to the conditional expectation appearing in
Equation~\eqref{eq:cont_value}. Specifically, by focusing on a single exercise
date $t_k$, we can express the continuation value \eqref{eq:cont_value}
as a conditional expectation of a discounted value computed at the subsequent
time step. To simplify notation within the kernel-based framework, we consider
a generic pair of random variables $(X, Y)$, where
\[
X \equiv S_k, \qquad Y \equiv S_{k+1}.
\]
The function $f$ defined on $\Ycal$ is used to represent the quantity
\begin{equation*}
    f(Y) = \Vcal_{k+1},
\end{equation*}
where $\Vcal_{k}$ is defined by recursion in 
Equation~\eqref{eq:backward_recursion_intro}. Under this identification, the
conditional expectation in Equation~\eqref{eq:cont_value} can be rewritten in
the generic CME notation as in Eq.~\eqref{eq:cme} up to the choice of measure. 
Each sample pair $(x_i, y_i)$ corresponds to two consecutive realizations of
the asset price along the same simulated path, i.e.,
\[
x_i = S_k^{(i)}, \qquad y_i = S_{k+1}^{(i)}, \quad i = 1, \dots, n.
\]
Moreover, since the underlying asset process $S_k$ is Markovian, the filtration
$\mathcal{F}_{k}$ can be identified with the $\sigma$-algebra generated by 
$S_k$. Hence, conditioning on $\mathcal{F}_{k}$ reduces to conditioning on
$S_k = x$, allowing us to express the continuation value estimation problem
entirely in terms of the CME operator
$\tilde C(t_k, S_k) = \tilde\E[f(Y)\mid X = x]$.
\end{remark}

Before proceeding, we briefly recall the definition of a reproducing kernel
Hilbert space. Let $\Hcal_\Xcal$ be a Hilbert space of functions on
$\Xcal$. 
A function $k_\Xcal\colon \Xcal \times \Xcal \to \mathbb{R}$ is called a
reproducing kernel of $\Hcal_\Xcal$ if $k_\Xcal(x,\cdot) \in \Hcal_\Xcal$
for all $x \in \Xcal$, and
\[
    f(x) = \langle f, k_\Xcal(x,\cdot) \rangle_{\Hcal_\Xcal}, 
    \quad \text{for all } f \in \Hcal_\Xcal,\ x \in \Xcal.
\]
Any Hilbert space admitting such a kernel is called a 
reproducing kernel Hilbert space (RKHS). 
Equivalently, there exists a feature map $\Phi_X\colon \Xcal \to \Hcal_\Xcal$ 
such that
\[
    k_\Xcal(x,x') = \langle \Phi_X(x), \Phi_X(x') \rangle_{\Hcal_\Xcal}, 
    \quad x,x' \in \Xcal.
\]

We assume the function space of interest for our framework to be a tensor
product RKHS denoted by $\Hcal = \Hcal_\Xcal \otimes \Hcal_\Ycal$, where 
$\Hcal_\Xcal$ and $\Hcal_\Ycal$ are separable RKHS associated with 
$\Xcal$ and $\Ycal$, respectively. The corresponding reproducing kernels are
denoted by $k_\Xcal$ and $k_\Ycal$. Consequently, the reproducing kernel $k$ 
of $\Hcal$ satisfies 
\[
    k\big((x,y),(x',y')\big) = k_\Xcal(x,x') k_\Ycal(y,y'),
    \quad x,x'\in\Xcal,\ y,y'\in\Ycal.
\]
For convenience, we define the row vectors of canonical feature maps as
\begin{equation}\label{PhiXYdef}
    \begin{aligned}
        {\bs\Phi}_X(\cdot) &
        \isdef [k_\Xcal(x_1,\cdot), \dots, k_\Xcal(x_n,\cdot)],\\
        {\bs\Phi}_Y(\cdot) &
        \isdef [k_\Ycal(y_1,\cdot), \dots, k_\Ycal(y_n,\cdot)],
    \end{aligned}
\end{equation}
and the associated kernel matrices 
\[
    \bs K_X \isdef [{\bs\Phi}_X(x_i)]_{i=1}^n, 
    \quad \bs K_Y \isdef [{\bs\Phi}_Y(y_i)]_{i=1}^n.
\]
Having introduced the necessary notation and preliminary concepts, we are now
ready to present the CME. We use the tensor product RKHS
$\Hcal = \Hcal _\Xcal \otimes \Hcal _\Ycal$, in order to estimate the
conditional expectation operator \eqref{eq:cme} for any $x \in \Xcal$. The
operator $\mu _{Y|X=x} \in \Hcal _\Ycal$ acts as a linear functional on
$\Hcal _\Ycal$, satisfying
\begin{equation}\label{CMEeq1}
    \langle f, \mu _{Y|X=x} \rangle _{\Hcal _{\Ycal}} 
    = \int _{\Ycal} f \dd \mathbb{Q}_{Y|X=x},
    \quad\text{for all } f \in \Hcal _\Ycal.
\end{equation}
Given a sample $(x_i, y_i)$, $i = 1, \dots, n$, and assuming the existence of
an element $\mu_{Y\mid X=\cdot} \in \Hcal$ satisfying \eqref{CMEeq1} for all
$x \in \Xcal$, \citet{gruenewaelderetal12} establish that the optimal estimator
among functions $\mu \in \Hcal$, $\mu: \Xcal \to \Hcal _\Ycal$, is obtained by
solving
\begin{equation}\label{eq:traditional}
    \hat{\mu} _{Y|X=\cdot} \isdef \argmin _{\mu \in \Hcal} \bigg\{ 
    \frac{1}{n} \sum _{i=1}^n \| k_{\Ycal}(y_i,\cdot) 
  - \mu(x_i) \|^2 _{\Hcal _\Ycal} + \lambda \| \mu \|_{\Hcal}^2 \bigg\}.
\end{equation}
The minimizer of \eqref{eq:traditional} is known to take the bilinear form
\[\hat{\mu} _{Y|X=\cdot} = {\bs\Phi}_Y(\cdot) {\bs F} {\bs\Phi}_X(\cdot)^{\T},
\]
where ${\bs\Phi}_X(\cdot)$ and ${\bs\Phi}_Y(\cdot)$ are defined in 
\eqref{PhiXYdef}, as established by \citet{micchellipontil05}.

By substituting this form into \eqref{eq:traditional}, the minimizer has the
representation
\begin{equation}\label{eq:conditionaldistributionembedding}
    \bm{F} = (\bm{K}_X + n \lambda \bm{I}_n)^{-1}, 
    \quad \text{and } \quad \hat{\mu} _{Y|X=\cdot} 
    = {\bs\Phi}_Y(\cdot) (\bm{K}_X 
    + n \lambda \bm{I}_n)^{-1} {\bs\Phi}_X(\cdot)^{\T}.
\end{equation}
The corresponding estimator for the conditional expectation operator
\eqref{CMEeq1} is given by
\begin{equation}\label{CMEeq1hat}
    \langle f, \hat{\mu}_{Y|X=x} \rangle _{\Hcal _{\Ycal}} 
    = [f(y_1), \dots, f(y_n)] \bs F {\bs\Phi}_X^{\T}(x), 
    \quad\text{for all } f \in \Hcal _\Ycal.
\end{equation}
A key observation from \eqref{CMEeq1hat} is that the optimal CME does not
explicitly involve the kernel matrix $\bm{K}_Y$, instead, $\Hcal_\Ycal$ is only
used for function evaluations.

To better understand the structure of the minimizer
\eqref{eq:conditionaldistributionembedding}, it is instructive to rewrite the
optimization problem \eqref{eq:traditional} in terms of the coefficient matrix
$\bm F$. By multiplying the objective function by $n$ and neglecting terms that
do not depend on $\bm F$, we obtain the equivalent finite-dimensional problem
\begin{equation*}
    \argmin _{\bs F \in \R^{n\times n}} \ 
    \Rcal_\lambda^{\textnormal{CME}}(\bs F),
\end{equation*}
where the objective function is given by
\begin{equation}\label{eq:condembobj}
	\begin{aligned}
		\Rcal_\lambda^{\textnormal{CME}}(\bs F) &= 
		\sum_{i=1}^n \left\{ -2 \, \bm \Phi_Y(y_i) \bm F \bm \Phi_X^{\T}(x_i)
		+ \bm \Phi_X^{\T}(x_i) \bm F^{\T} \bm K_Y 
  \bm F \bm \Phi_X^{\T}(x_i) \right\} \\
		& \quad + n \lambda \, \text{trace} 
    \left( \bm F^{\T} \bm K_Y \bm F \bm K_X \right) \\
		&= -2 \, \text{trace} \left( \bm K_Y \bm F \bm K_X \right) 
		+ \text{trace} \left( \bm F^{\T} \bm K_Y \bm F \bm K_X \bm K_X \right) \\
		& \quad + n \lambda \, \text{trace} 
    \left( \bm F^{\T} \bm K_Y \bm F \bm K_X \right) \\
		&= -2 \, \text{trace} \left( \bm K_Y \bm F \bm K_X \right)
		+ \text{trace} \left( \bm F^{\T} \bm K_Y \bm F \bm K_X 
    \left( \bm K_X + n \lambda \bm I_n \right) \right).
	\end{aligned}
\end{equation}
From the first-order optimality condition with respect to the matrix $\bm F$,
we obtain
\[
-2 \bm K_Y \bm K_X + 2 \bm K_Y \bm F 
\left( \bm K_X + n \lambda \bm I_n \right) \bm K_X = \bm 0,
\]
which yields
$
\bm F = (\bm K_X + n \lambda \bm I_n)^{-1}
$,
as anticipated in Eq.~\eqref{eq:conditionaldistributionembedding}.

A major computational challenge arises when working with large datasets, as the
inversion of the matrix in \eqref{eq:conditionaldistributionembedding} becomes
computationally expensive, with cubic complexity in the sample size. 
This limitation motivates the development of scalable approximations of the
conditional mean embedding. We introduce here a low-rank approximation aimed at
alleviating the computational bottlenecks of the full-rank formulation. 
Building on the matrix representation derived above, we exploit structured
low-rank decompositions of the kernel matrices to obtain a reduced
representation of the conditional expectation operator. 

\subsection{Pivoted Cholesky and double-orthogonal basis}\label{sec:pivchol}
We employ an efficient low-rank representation of the kernel matrix,
considered before in \citet{FilipovicMultererSchneider25}. In particular, we
adopt an adaptive low-rank approach based on the
\emph{pivoted Cholesky decomposition} for the kernel matrices $\bs K_X$ and
$\bs K_Y$. The method constructs, in a single procedure, both a low-rank
Cholesky factorization of the kernel matrix and an associated
\emph{double-orthogonal basis transformation}. This approach allows for the
derivation of an approximation that is orthogonal with respect to both the 
RKHS $\Hcal$ and {the empirical $L^2$ space associated with the sampling 
measure $\tilde{\Q}$.}
The resulting transformation diagonalizes the quadratic terms in the objective
function, substantially simplifying the ensuing optimization problem.

The computational procedure follows the algorithmic structure detailed in
\citet[Algorithm~4.1]{FilipovicMultererSchneider25}. For the reader's
convenience, we recall it in Algorithm~\ref{algo:pivChol}.

\begin{algorithm}[htb]
\caption{Pivoted Cholesky decomposition}
\label{algo:pivChol}
\begin{flushleft}
\begin{tabular}{ll}
\textbf{input:}  & symmetric and positive semidefinite
matrix ${\boldsymbol K}\in\mathbb{R}^{N \times N}$, 
\(\varepsilon\geq0\)\\
\textbf{output:} & low-rank approximation
\({\boldsymbol K}\approx{\boldsymbol L}{\boldsymbol L}^\top\)\\
&and
biorthogonal basis \({\boldsymbol B}\)
such that \({\boldsymbol B}^\top{\boldsymbol L}={\boldsymbol I}_m\)
\end{tabular}
\end{flushleft}
\begin{algorithmic}[1]
\smallskip
\State Initialization: set $m\isdef 1$,
 ${\boldsymbol d}\isdef \operatorname{diag}({\boldsymbol K
})$, \({\boldsymbol L}\isdef[\,]\),
 \({\boldsymbol B}\isdef[\,],\)
$\operatorname{err}\isdef \|{\boldsymbol d}\|_{1}$
\State\textbf{while} \(\operatorname{err}>\varepsilon\)
\State\StateIndent determine
$\pi(m)\isdef \operatorname{arg}\max_{1\le i\le N} d_i$
\State\StateIndent compute \[
\boldsymbol\ell_m
\isdef \frac{1}{\sqrt{d_{\pi(m)}}}\Big({\boldsymbol K}-{\boldsymbol L}
{\boldsymbol L}^\top\Big)\boldsymbol e_{\pi(m)}\quad\
\text{and}\quad\boldsymbol b_m
\isdef \frac{1}{\sqrt{d_{\pi(m)}}}\Big({\boldsymbol I}-{\boldsymbol B}
{\boldsymbol L}^\top\Big)\boldsymbol e_{\pi(m)}
\]
\State\StateIndent set \({\boldsymbol L}
\isdef [{\boldsymbol L},{\boldsymbol\ell}_m], \quad {\boldsymbol B}
\isdef [{\boldsymbol B},{\boldsymbol b}_m]\)
\State\StateIndent set \({\boldsymbol d}
\isdef {\boldsymbol d}-{\boldsymbol\ell}_m\odot
\,{\boldsymbol\ell}_m\), where $\odot$ denotes the Hadamard product
\State\StateIndent set \(\operatorname{err}
\isdef\|{\boldsymbol d}\|_1\), \(m\isdef m+1\)
\end{algorithmic}
\end{algorithm}

For a given symmetric
positive semidefinite matrix $\bs K\in \R^{n\times n}$, the algorithm
iteratively constructs an incomplete Cholesky factorization
$\bs K \approx \bs L \bs L^{\T}$, where $\bs L \in \R^{n\times m}$, up to a
prescribed tolerance $\varepsilon$, while simultaneously determining a
biorthogonal basis transformation $\bs B \in \R^{n\times m}$.
The pivot is selected greedily as the entry corresponding to the largest
diagonal element of the current Schur complement, in line with the classical
strategy of \citet{HPS12}.
The theoretical validity of the algorithm is established by the following
result.

\begin{theorem}[{\citet[Theorem~4.1]{FilipovicMultererSchneider25}}]
\label{thmChol}
For any tolerance $\varepsilon > 0$, given $\bs K\in \R^{n\times n}$,
Algorithm~\ref{algo:pivChol} computes matrices
$\bm B, \bm L \in \R^{n\times m}$ with $m \leq \operatorname{rank}(\bs K)$, such
that $\bm K - \bm L \bm L^\T$ is positive semidefinite and satisfies,
\begin{align*}
\operatorname{trace}\big({\bs K}-{\bs L}{\bs L}^\T\big)&\leq\varepsilon,\\
\operatorname{Im}\bm B &= 
\operatorname{span}\{{\bs e}_{p_1},\dots,{\bs e}_{p_m}\}, \\
\bm B^\T \bm L &= {\bm I},\\
\bm K \bm B &= \bm L.
\end{align*}
\end{theorem}
The pivoted Cholesky factorization enables a significant reduction in
computational cost for the conditional mean embedding, given that the trace
of the Schur complements is reduced sufficiently fast.
In particular, it reduces the overall cost to $\Ocal(n m^2)$, the memory 
requirements to $\Ocal(nm)$, and the evaluation cost to $\Ocal(nm)$, while never
requiring explicit assembly of the full kernel matrix. Only pivot columns and
diagonal elements are computed, ensuring scalability even for large datasets.

Based on the biorthogonal transformation $\bs B$, an additional rotation can be
introduced to diagonalize the quadratic terms of the regularized objective function. 
This is achieved by performing a spectral decomposition 
\[
\bs V \bs \Lambda \bs V^{\T} = \bs L^{\T} \bs L,
\]
with $\bs V, \bs \Lambda \in \R^{m\times m}$, incurring a computational cost of
order $\Ocal(m^3)$. 
Defining the transformed basis as
\begin{equation}\label{eq:spectralQ}
    \bs Q \isdef \bs B \bs V,
\end{equation}
we obtain a double-orthogonal representation that simplifies subsequent 
optimization steps. Since there holds 
\[
  \big\langle({\bs\Phi}{\bs Q})^\top,{\bs\Phi}{\bs Q}\big\rangle_\Hcal
  ={\bs V}^\top{\bs B}^\top{\bs K}{\bs B}{\bs V}={\bs I}
\]
as well as
\[
  \big\langle({\bs\Phi}{\bs Q})^\top,{\bs\Phi}{\bs Q}
  \big\rangle_{L^2_{\widehat{\mathbb{P}}}}
  ={\bs V}^\top{\bs B}^\top{\bs K}^2{\bs B}{\bs V}={\bs \Lambda},
\]
where \(\widehat{\mathbb{P}}\) is the sample measure associated to the
data sites in the canonical feature vector ${\bs\Phi}$, the obtained basis can
be considered a discretized version of the kernels spectral basis.

\subsection{Low-rank approximation of CME}
\label{sec:conditionaldistemblowrank}
We now apply the low-rank construction introduced in Section~\ref{sec:pivchol}
to the conditional mean embedding framework. In particular, we consider low-rank
approximations of the kernel matrices $\bs K_X$ and $\bs K_Y$ based on the
pivoted Cholesky decomposition and the associated double-orthogonal basis
transformation.

Applying pivoted Cholesky and the transformation \eqref{eq:spectralQ}
independently to $\bs K_X$ and $\bs K_Y$ yields the decompositions
\[
\bs K_X \approx \bs L_X \bs L_X^\T, \qquad 
\bs K_Y \approx \bs L_Y \bs L_Y^\T,
\]
together with the spectral decompositions
\[
\bs V_X \bs \Lambda_X \bs V_X^\T = \bs L_X^\T \bs L_X, \qquad
\bs V_Y \bs \Lambda_Y \bs V_Y^\T = \bs L_Y^\T \bs L_Y,
\]
and the corresponding basis matrices
\[
\R^{n \times m_X} \ni \bs Q_X = \bs B_X \bs V_X, \qquad 
\R^{n \times m_Y} \ni \bs Q_Y = \bs B_Y \bs V_Y,
\]
which define reduced representations of the feature spaces associated with 
$X$ and $Y$, respectively.
These constructions allow us to reformulate the CME estimation problem in a
lower-dimensional subspace, avoiding operations involving the full kernel
matrices. To this end, we exploit the matrix formulation of the objective
function introduced in \eqref{eq:condembobj} and seek a reduced representation
of the coefficient matrix $\bm F$ that is compatible with the low-rank
structure. In particular, we approximate $\bm F$ by restricting it to the
subspace spanned by the low-rank bases associated with $X$ and $Y$. 
This leads to a parametrization of the form

\begin{equation*}
	\bm F = \bm Q_Y \tilde{\bm F}  \bm Q_X^{\T}.
\end{equation*}
We insert this into the objective function \eqref{eq:condembobj},
with $\tilde{\bm F} \in \mathbb{R}^{m_Y \times m_X}$, and we obtain the
low-rank minimization problem on $\mathbb{R}^{m_Y \times m_X}$ according to

\begin{equation*}
    \argmin _{\tilde{\bm F} \in \R^{m_Y \times m_X}} 
    \ \tilde{\Rcal}_\lambda^{\textnormal{CME}}(\tilde{\bm F}) ,
\end{equation*}
where

\begin{equation}\label{eq:lowrankprog}
	\begin{aligned}
		\tilde{\Rcal}_\lambda^{\textnormal{CME}}(\tilde{\bm F}) &= 
		\Rcal_\lambda^{\textnormal{CME}}(\bm Q_Y \tilde{\bm F}  \bm Q_X^{\T}) \\
		&= -2 \, \text{trace} \left( \bm L_Y \bm V_Y \tilde{\bm F} \bm V_X^\T 
    \bm L_X^{\T} \right)+ \text{trace} \left( \tilde{\bm F} 
  \bm \Lambda_X \tilde{\bm F}^\T 
+ n \lambda \tilde{\bm F} \tilde{\bm F}^\T \right).
	\end{aligned}
\end{equation}
The reduced objective function \eqref{eq:lowrankprog} is quadratic in
$\tilde{\bm F}$ and admits a closed-form solution. By computing the first-order
optimality condition with respect to $\tilde{\bm F}$, we obtain

\begin{equation}\label{eq:Ftilde}
    \tilde{\bm F} = \left( \bm L_Y \bm V_{Y} \right)^\T
    \left( \bm L_X \bm V_{X} \right) \left( \bm \Lambda_X 
    + n \lambda \bm I_{m_X} \right)^{-1},
\end{equation}
and

\begin{equation}\label{eq:sol_lowrank}
    \tilde{\mu} _{Y|X=\cdot} = {\bs\Phi}_Y(\cdot) \ 
    \bs Q_Y \tilde{\bm F} \bs Q_X^\T \  {\bs\Phi}_X(\cdot)^{\T}.
\end{equation}
Which is an approximation of the CME in terms of the discrete spectral bases
associated to \(k_\Xcal\) and \(k_\Ycal\).
For all $f \in \Hcal _\Ycal$, the corresponding estimator for the conditional 
expectation operator \eqref{CMEeq1} is given by
\begin{equation}\label{eq:sol_lowrank2}
    \langle f, \tilde{\mu}_{Y|X=x} \rangle _{\Hcal _{\Ycal}} 
    =  [f(y_1), \dots, f(y_n)] \bs Q_Y \tilde{\bm F} 
    \bm Q_X^\T {\bs\Phi}_X^{\T}(x).
\end{equation}
The computational cost to obtain $\tilde{\bm F} \in \mathbb{R}^{m_Y \times m_X}$
is $\mathcal{O}(n \, m_X \,  m_Y)$. Afterward, due to the sparsity of 
$\bm B_X \in \mathbb{R}^{n \times m_X}$ and 
$\bm B_Y \in \mathbb{R}^{n \times m_Y}$, the evaluation of the conditional mean
embedding can be performed at a cost of $\mathcal{O}(m_X \, m_Y)$.
Observe that, in contrast to the full-rank formulation in \eqref{CMEeq1hat},
where the kernel matrix $\bm{K}_Y$ does not appear explicitly and
$\mathcal{H}_\Ycal$ is only required for function evaluations, the low-rank
approximation introduces an explicit dependence on the finite-dimensional
structure of $\mathcal{H}_\Ycal$.

A concise comparison between the full-rank and low-rank formulations of
conditional mean embedding is reported in Table~\ref{tab:cme_summary},
highlighting the main differences in objective functions, solutions and
predictive forms.

\begin{table}[ht]
\centering
\renewcommand{\arraystretch}{2}
\begin{tabular}{| >{\centering\arraybackslash}m{7cm} | 
  >{\centering\arraybackslash}m{7cm} |}
\hline
\textbf{Full-rank CME} & \textbf{Low-rank CME} \\
\hline
$\argmin_{\bm F \in \mathbb{R}^{n \times n}} \ 
\Rcal_\lambda^{\textnormal{CME}}(\bm F)$
& $ \argmin_{\tilde{\bm F} \in \mathbb{R}^{m_Y \times m_X}} \ 
\tilde{\Rcal}_\lambda^{\textnormal{CME}}(\tilde{\bm F})$ \\
\hline
$\displaystyle \Rcal_\lambda^{\textnormal{CME}}(\bm F) = -2\, 
\text{tr}(\bm K_Y \bm F \bm K_X) + \text{tr}(\bm F^\T \bm K_Y 
\bm F \bm K_X (\bm K_X + n \lambda \bm I_n))$
& $\displaystyle \tilde{\Rcal}_\lambda^{\textnormal{CME}}(\tilde{\bm F}) = -2\, 
\text{tr}(\bm L_Y \bm V_Y \tilde{\bm F} \bm V_X^\T \bm L_X^\T) 
+ \text{tr}(\tilde{\bm F} \bm \Lambda_X \tilde{\bm F}^\T 
+ n \lambda \tilde{\bm F} \tilde{\bm F}^\T)$ \\
\hline
$\displaystyle \hat{\mu}_{Y|X=\cdot} 
= {\bs\Phi}_Y(\cdot) \bm F {\bs\Phi}_X(\cdot)^\T$
& $\displaystyle \tilde{\mu}_{Y|X=\cdot} = {\bs\Phi}_Y(\cdot) 
\bs Q_Y \tilde{\bm F} \bs Q_X^\T {\bs\Phi}_X(\cdot)^\T$ \\
\hline
$\displaystyle \bm F = (\bm K_X + n \lambda \bm I_n)^{-1}$
& $\displaystyle \tilde{\bm F} = (\bm L_Y \bm V_Y)^\T (\bm L_X \bm V_X) 
(\bm \Lambda_X + n \lambda \bm I_{m_X})^{-1}$ \\
\hline
$\displaystyle \langle f, \hat{\mu}_{Y|X=x} \rangle_{\Hcal_\Ycal} 
= [f(y_1), \dots, f(y_n)] \bm F {\bs\Phi}_X^\T(x)$
& $\displaystyle \langle f, \tilde{\mu}_{Y|X=x} \rangle_{\Hcal_\Ycal} 
= [f(y_1), \dots, f(y_n)] \bs Q_Y \tilde{\bm F} \bs Q_X^\T {\bs\Phi}_X^\T(x)$ \\
\hline
\end{tabular}
\caption{Comparison of full-rank and low-rank CME formulations.}
\label{tab:cme_summary}
\end{table}

\section{Convergence rate for low-rank CME}\label{sec:conv_rate}
In this section, we study the convergence properties of the proposed low-rank 
conditional mean embedding.
By the triangle inequality, the total error admits the decomposition
\begin{equation}\label{eq:tot_err}
\|\mu - \tilde{\mu}\|_{\Hcal} \leq 
\underbrace{\|\mu - \hat{\mu}\|_{\Hcal}}_{\text{statistical error}} 
+ 
\underbrace{\|\hat{\mu} - \tilde{\mu}\|_{\Hcal}}_{\text{low-rank approximation error}},
\end{equation}
where $\mu$ is the CME, $\hat{\mu} \isdef \hat{\mu}_{Y|X=\cdot}$ 
the empirical estimator defined by \eqref{eq:conditionaldistributionembedding}, 
and $\tilde{\mu} \isdef \tilde{\mu}_{Y|X=\cdot}$ the low-rank approximation 
given in \eqref{eq:sol_lowrank}.
In this paper, we focus on the second term, the low-rank approximation error, 
which quantifies the additional error introduced by replacing the full-rank 
estimator $\hat{\mu}$ with its low-rank counterpart $\tilde{\mu}$. 
In order to analyze this error, we now characterize 
how the empirical estimator $\hat{\mu}$ can be projected onto the reduced 
space $\tilde{\Hcal}$, defined as
\begin{equation*}
		\tilde\Hcal \isdef  \spn \big\{ \left( \bs\Phi_{X} \bm Q_X  \right)_i 
      \otimes \left( \bs\Phi_{Y} \bm Q_Y  \right)_j: 
		i=1,\dots,m_X,\, j=1,\dots,m_Y\big\},
\end{equation*}
and how this projection relates to the actual low-rank solution $\tilde{\mu}$.
This leads to the following result.

\begin{theorem}[{Low-rank approximation error}]\label{thm:approx}
    Let \(\hat{\mu} \) be the empirical estimator defined by 
    \eqref{eq:conditionaldistributionembedding}, and \(\tilde{\mu} \) 
    the low-rank approximation given in \eqref{eq:sol_lowrank}, computed 
    under the low-rank approximation conditions 
	\(\trace\big({\bs K}_X-{\bs L}_X{\bs L}_X^\T\big)<\varepsilon\) and 
	\(\trace\big({\bs K}_Y-{\bs L}_Y{\bs L}_Y^\T\big)<\varepsilon\). 
  Then the orthogonal projection of $\hat \mu$ onto $\tilde\Hcal$ is given by
	\[
	\hat \mu_{\tilde\Hcal}= \bs\Phi_Y (\cdot) \bm Q_Y  \bm F_{\tilde\Hcal} 
	\bm Q_X^\T \bs\Phi _X (\cdot)^{\T},
	\]
	with coefficient matrix
	\begin{equation}\label{eq:orthoproj}
		\bm F_{\tilde\Hcal} = \argmin_{\bs F'\in 
			\R ^{m_Y\times m_X}} {\big\| \hat \mu -{\bs\Phi}_Y(\cdot) \bm Q_Y  \bs F' 
			 	\bm Q_{X}^\T  {\bs\Phi}_X(\cdot)^\T}\big\|_{\Hcal}^2
        =(\bm L_Y \bm V_Y)^\T \bm F\bm L_{X}\bm V_{X}.
	\end{equation}
	Moreover, the following inequality holds,
	\begin{equation}\label{eq:err_lr}
        {\| \hat \mu -\tilde \mu \|^2_{\Hcal}}
	\leq \varepsilon  \, \|{\bs F}\|_F^2 \big(\trace({\bs K}_X)
  +\trace({\bs K}_Y)\big) 
	+{\| \bm F_{\tilde\Hcal}-\tilde{\bm F} \|_F^2},
	\end{equation}
    where \(\|\cdot\|_F\) denotes the Frobenius norm.
\end{theorem}
Refer to Appendix~\ref{app:proof_lowrank} for a detailed proof of 
Theorem~\ref{thm:approx}.
Theorem~\ref{thm:approx} highlights the role of the subspace $\tilde\Hcal$ in
controlling the low-rank approximation error. It is therefore natural to further
investigate the behavior of the associated coefficient matrices, which are
summarized in Table~\ref{tab:coeff-matrices}. To this end, 
we note that
\begin{equation*}
    \tilde{\bm F} = (\bm L_Y \bm V_Y)^\T (\bm L_X \bm V_X)
(\bm \Lambda_X + n \lambda \bm I_{m_X})^{-1}=
(\bm L_Y \bm V_Y)^\top 
    (\bm L_X \bm L_X^\top + n\lambda \bm I_n)^{-1} \bm L_X \bm V_X.
\end{equation*}

\begin{table}[ht]
\centering
\renewcommand{\arraystretch}{1.4}
\begin{tabular}{|l|l|}
\hline
{Empirical estimator $\hat{\mu} \in \Hcal$} 
& $\displaystyle \hat{\mu} = {\bs\Phi}_Y(\cdot) \bm F {\bs\Phi}_X(\cdot)^\T$ \\
& $\displaystyle \bm F = (\bm K_X + n \lambda \bm I_n)^{-1}$ \\ 
\hline

{Low-rank approximation $\tilde{\mu} \in \tilde\Hcal$} 
& $\displaystyle \tilde{\mu} = {\bs\Phi}_Y(\cdot)  \bm Q_Y \tilde{\bm F} 
\bm Q_X^\T {\bs\Phi}_X(\cdot)^\T$ \\ 
& $\displaystyle \tilde{\bm F} = (\bm L_Y \bm V_Y)^\T (\bm L_X \bm V_X)
(\bm \Lambda_X + n \lambda \bm I_{m_X})^{-1}$ \\ 
\hline

{Orthogonal projection $\hat{\mu}_{\tilde\Hcal}$ of $\hat \mu$ in $\tilde\Hcal$} 
& $\displaystyle \hat \mu_{\tilde\Hcal} = \bs\Phi_Y (\cdot) \bm Q_Y 
\bm F_{\tilde\Hcal} \bm Q_X^\T \bs\Phi_X (\cdot)^\T$ \\ 
& $\displaystyle \bm F_{\tilde\Hcal} = (\bm L_Y \bm V_Y)^\T 
\bm F \bm L_X \bm V_X$ \\ 
\hline
\end{tabular}
\caption{Summary of the different coefficient matrices used in the analysis.}
\label{tab:coeff-matrices}
\end{table}

The latter representation is particularly convenient 
because it depends directly on the inverse of the regularized 
low-rank matrix $(\bm L_X \bm L_X^\top + n\lambda \bm I_n)$, which simplifies
the estimate of the last term in Equation~\eqref{eq:err_lr},  i.e. 
$\| \bm F_{\tilde\Hcal} - \tilde{\bm F} \|_F^2$,
leading to the following result.

\begin{proposition}[Estimation of the coefficient error]
\label{prop:coeff_err}
Let $\bm F_{\tilde\Hcal}$ denote the orthogonal projection of the empirical
coefficient matrix onto the low-rank subspace $\tilde{\Hcal}$, and let
$\tilde{\bm F}$ be the corresponding low-rank solution. Given $\lambda >0$,
the Frobenius norm of the difference between these coefficient matrices
satisfies
\[
\big\| \bm F_{\tilde\Hcal} - \tilde{\bm F}\big\|_F^2
\le
\frac{\varepsilon^2}{{(n\lambda)^4}}\operatorname{trace}(\bm K_X)
\operatorname{trace}(\bm K_Y).
\]
\end{proposition}
Refer to Appendix~\ref{app:proof_coeff} for a detailed proof of 
Proposition \ref{prop:coeff_err}.

Combining Theorem~\ref{thm:approx} and Proposition~\ref{prop:coeff_err}, 
we obtain an explicit bound on the low-rank approximation error:
\begin{equation}\label{eq:delta_LR}
\|\hat{\mu} - \tilde{\mu}\|^2_{\Hcal}
\,\leq\,
\delta^{LR}
\end{equation}
with
\begin{equation*}
\delta^{LR} = \varepsilon \, \|\bm F\|_F^2 \bigl(\operatorname{trace}(\bm K_X) 
+ \operatorname{trace}(\bm K_Y)\bigr)
+
\frac{\varepsilon^2}{(n\lambda)^4} \,
\operatorname{trace}(\bm K_X)\operatorname{trace}(\bm K_Y).
\end{equation*}

Turning to the statistical error term in \eqref{eq:tot_err},
deriving such a bound is a well-studied problem in the literature, and several
convergence results have been established under various assumptions on the
regularity of the CME and the kernel. In this work, we rely on the sharp rates
derived in \citet{LiMeunier_etal2024}. In particular, this term satisfies, up to
logarithmic factors and with probability at least $1 - 5e^{-\tau}$ that
\begin{equation}\label{eq:stat_err_Li}
\|\mu - \hat{\mu}\|^2_\gamma \leq \delta^{S}_\gamma ,
\end{equation}
with
\begin{equation*}
\delta^{S}_\gamma=  \tau^2 c_1 \, 
n^{-\frac{\beta - \gamma}{\max(1,\, \beta+p)}}.
\end{equation*}
Here $\gamma \in [0,1]$ is a norm interpolation parameter: $\gamma = 0$
corresponds to the $L^2$ norm, while $\gamma = 1$ corresponds to the 
RKHS norm. The constant $c_1 > 0$ is independent of $n$ and $\tau > \log(5)$
is a confidence parameter. The parameter $\beta > 0$ controls the smoothness of
the true CME $\mu$: the case $\beta \geq 1$ corresponds to the well-specified
setting where $\mu \in \Hcal$, while $\beta < 1$ corresponds to the misspecified
setting where $\mu$ lies in a larger interpolation space. 
The parameter $p \in (0,1]$ governs the eigenvalue decay of the integral
operator associated with the kernel $k_\Xcal$: smaller $p$ corresponds to faster
decay and a simpler approximation problem.
A detailed discussion of the relevant literature, the required assumptions, 
and a precise statement of the bound are provided in
Appendix~\ref{app:stat_error}.

\section{Offline training and pricing backward error}
\label{sec:offline_cme}
We now aim to derive an upper bound for the error arising during the backward
recursion of the optimal stopping problem, where at each time step the
continuation value is approximated through the CME. In contrast to the previous
section, where we analyzed the CME approximation error at a fixed time step,
we now fix the sample path and investigate how this error evolves backward in
time throughout the pricing procedure.

A similar analysis for the Longstaff-Schwartz method has been carried out by
Zanger in a series of papers 
\citep{Zanger2009123,Zanger2013503,Zanger2018447,Zanger2020923}, in which the
backward error propagation is investigated under increasingly general
assumptions on the approximating class.
A fundamental distinction between the proposed CME-based approach and the
classical Longstaff-Schwartz method concerns the temporal structure of the
approximation problem arising in the backward recursion.
In the Longstaff--Schwartz framework, the continuation value at time $t_k$ is
approximated by solving, at each backward step, a regression function of the
form $ \sum_{j=1}^{J} \beta_{k,j} \, \phi_j(S_k)$,
where the regression target depends on the optimal stopping decisions at all
future times.
As a consequence, the regression coefficients $\beta_{k,j}$ are inherently
time-dependent and must be recomputed at every exercise date during the
backward induction.

In contrast, the CME-based approach estimates a conditional expectation operator
\begin{equation}\label{eq:cme_op}
    f \mapsto \mathbb{E}[ \, f(S_{k+1}) \mid S_k = \cdot \, ],
\end{equation}
which acts linearly on functions $f$ defined on the state space.
Under the standing assumption that the underlying process
$(S_k)_{k=0,\dots,{n_T}}$ is Markovian and time-homogeneous, the conditional
distribution of $S_{k+1}$ given $S_k$ is independent of the time index $k$.
Consequently, the operator \eqref{eq:cme_op} does not depend on the exercise 
date. This observation has important computational and conceptual
implications. Once a CME approximation \eqref{eq:cme_op} of the conditional
expectation operator has been constructed from simulated sample pairs
$(S_k,S_{k+1})$, it can be reused across all backward steps of the dynamic
programming recursion. At each time step, the continuation value is obtained by
applying the same operator \eqref{eq:cme_op} to a different function $f$,
corresponding to the (possibly approximated) value function at the next time
step. In particular, the CME-based approach allows for an
\emph{offline training} phase in which the conditional expectation operator is
estimated once, followed by an \emph{online evaluation} phase in which the
backward recursion is performed without re-estimating regression coefficients.

Given \(P_k\) the square-integrable payoff process of the derivative,
and recalling that the underlying process
\(\{S_k\}_{k=0,\dots,n_T}\) is Markovian, the optimal stopping problem can be
expressed as the search for the value process
\begin{align*}
	\mathcal{V}_k
	\isdef
	& \operatorname*{ess\,sup}_{\tau \in \mathcal{T}_{k,n_T}}
	\mathbb{E} \!\left[ P_\tau \,\middle|\, S_k \right],
\end{align*}
where \(\mathcal{T}_{k,n_T}\) denotes the set of all stopping times taking
values in \(\{k, \dots, n_T\}\).
A stopping time \(\tau_k \in \mathcal{T}_{k,n_T}\) is called \emph{optimal} if
\begin{equation}\label{eq:opt_cond}
    \mathcal{V}_k = \mathbb{E}[P_{\tau_k} \mid S_k].
\end{equation}
As introduced in Equation~\eqref{eq:backward_recursion_intro}, from standard
optimal stopping theory \citep{karatzas1998,glasserman04}, it follows that
\begin{equation}\label{eq:backward_recursion}
\begin{aligned}
	\mathcal{V}_{n_T} &= P_{n_T}, \\
  \mathcal{V}_k &= \max \big\{ P_k, \mathbb{E}[ \mathcal{V}_{k+1} 
  \mid S_k ] \big\},
	\quad k = 0, \dots, {n_T}-1. \end{aligned}
\end{equation}
Combining the recursion \eqref{eq:backward_recursion} for \(k=0\) with the
optimality condition \eqref{eq:opt_cond} yields
\begin{align*}
	\mathcal{V}_0
  &= \max \big\{ P_0, \mathbb{E}[ \mathcal{V}_1 \mid S_0 ] \big\}
	\\
  &= \max \big\{ P_0, 
  \mathbb{E}[ \mathbb{E}[ P_{\tau_1} \mid S_1 ] \mid S_0 ] \big\}
	\\
  &= \max \big\{ P_0, \mathbb{E}[ P_{\tau_1} \mid S_0 ] \big\},
\end{align*}
where the last equality follows from the tower property of conditional 
expectation.
Given that \(S_0\) is deterministic, it follows that
\begin{align}\label{eq:init_value}
	\mathcal{V}_0
	&= \mathbb{E}[ P_{\tau_0} ]
  = \max \big\{ P_0, \mathbb{E}[ P_{\tau_1} \mid S_0 ] \big\}
  = \max \big\{ P_0, \mathbb{E}[ P_{\tau_1} ] \big\}.
\end{align}
This recursive characterization provides the foundation for studying how the
approximation error in the CME evolves backward in time through the sequence of
conditional expectations in \eqref{eq:backward_recursion}.
Before deriving the error bound, we introduce some notation and recall a
fundamental property of the conditional expectation operator. 

For any square-integrable measurable function
$f \in L^2_k \isdef L^2(\Omega, \mathcal{F}_k, \mathbb{Q})$, we consider the
associated norm
\[
\| f \|_{L^2_k}^2 \isdef \mathbb{E}\left[ |f(S_k)|^2 \right].
\]
The conditional expectation operator, which maps a function \(f \in L^2_{k+1}\) 
to a function in \(L^2_k\) according to
\[
\mathbb{E}\left[ f(S_{k+1}) \mid S_k = x \right],
\]
is a contraction in the \(L^2\) sense, meaning that
\[
\big\| \mathbb{E}\left[ f(S_{k+1}) \mid S_k = \cdot \right] 
\big\|_{L^2_k} \le \| f \|_{L^2_{k+1}},
\]
which follows directly from Jensen's inequality applied to the conditional
expectation. Having established this property for the exact conditional
expectation operator, we now introduce its CME-based approximation, which will
be used to analyze the backward error propagation in the recursive scheme.
For each time step \(k = 0, \dots, {n_T}-1\), we denote by
\[
\tilde{\mathbb{E}}\left[ P_{\tilde{\tau}_{k+1}} \mid S_k = x \right]
\]
the low-rank CME approximation $\tilde \mu \in \tilde{\mathcal{H}}$ introduced
in Section \ref{sec:conditionaldistemblowrank}. Similarly, we obtain the CME
approximated analog of the recursion \eqref{eq:init_value}, namely
\begin{equation}\label{eq:init_value_approx}
\tilde{\mathbb{E}}\left[ P_{\tilde{\tau}_0} \right]
= 
\max\big\{ P_0, \tilde{\mathbb{E}}\left[ P_{\tilde{\tau}_1} 
\mid S_0 \right] \big\}
=
\max\big\{ P_0, \tilde{\mathbb{E}}\left[ P_{\tilde{\tau}_1} \right] \big\}.
\end{equation}
Observe that all the conditional expectations appearing in the backward
recursion, namely 
\(\mathbb{E}[ P_{\tau_k}\mid S_k ]\),
\(\mathbb{E}[ P_{\tau_{k+1}} \mid S_k ]\),
and their CME-based counterparts, belong to \(L^2_k\) for each \(k\).  
This follows from the square-integrability of \(P_k\) and from the
\(L^2\)-contraction property of the conditional expectation operator,
applied inductively backward in time.

In the CME framework, since the same operator is used at every exercise date, 
the error incurred at one step does not come from a new regression fit, but from 
the repeated application of the same offline-trained approximation.
To bound the local approximation error, we work in the $L^2$ norm, 
which is consistent with the classical framework used in the option pricing 
literature. For the statistical error, this corresponds to setting $\gamma = 0$ 
in \eqref{eq:stat_err_Li}, giving the bound $\delta^S_0$. For the low-rank 
approximation error, the bound \eqref{eq:delta_LR} is established in the 
$\Hcal$-norm. To transfer it to the $L^2$ norm, we use the standard embedding, 
\[
\|\hat{\mu} - \tilde{\mu}\|_{L^2} \leq 
\kappa_X \|\hat{\mu} - \tilde{\mu}\|_{\Hcal},
\]
where $\kappa_X = \sup_x \sqrt{k_\Xcal(x,x)}$ (observe that this is precisely
the boundedness condition assumed in \textnormal{(H3)} of
Appendix~\ref{app:stat_error}, where the same constant $\kappa_X$ appears).
We therefore obtain, for each time step $j = 0, \dots, n_T-1$, the local 
approximation error bound
\begin{equation}\label{eq:local_err}
\big\| \tilde{\mathbb{E}}[P_{\tilde{\tau}_{j+1}} \mid S_j]
- \mathbb{E}[P_{{\tilde \tau}_{j+1}} \mid S_j] \big\|_{L^2_j} 
\leq \sqrt{\delta^S_0}  +  \kappa_X \sqrt{\delta^{LR}}
\end{equation}
as shown in Equations~\eqref{eq:delta_LR} and~\eqref{eq:stat_err_Li}, 
and via the triangle inequality \eqref{eq:tot_err}.
Due to the time-homogeneity of the Markov process and the reuse of the same 
operator at each step, the bound in Eq.~\eqref{eq:local_err} is independent
of $j$.

Since the pricing algorithm with the low-rank CME approximation follows the same 
backward dynamic programming structure as the Longstaff-Schwartz algorithm, 
differing only in the way the continuation value is approximated at each step, 
the local backward error propagation result of \citet{Zanger2009123} applies 
directly to our setting. Indeed, the key properties required by
\citet{Zanger2009123}, namely the square-integrability of the value process and
the $L^2$-contraction property of the conditional expectation operator,
have been verified above for our CME-based approximation. We state the result of
Zanger here for completeness.

\begin{proposition}[{\cite[Lemma~2.2]{Zanger2009123}}]
\label{prop:local_err_prop}
For each \(k = 0, \dots, {n_T}-1\), the following inequality holds:
\begin{equation*}
\big\|\tilde{\mathbb{E}}[ P_{\tilde{\tau}_{k+1}} \mid S_k ]  
- \mathbb{E}[ P_{\tau_{k+1}} \mid S_k ]\big\|_{L^2_k} 
\le 2 \sum_{j=k}^{{n_T}-1} \big\|\tilde{\mathbb{E}}[ P_{\tilde{\tau}_{j+1}} \mid S_j ]  
- \mathbb{E}[ P_{ \tilde \tau_{j+1}} \mid S_j ] \big\|_{L^2_j}.
\end{equation*}
\end{proposition}

Proposition~\ref{prop:local_err_prop} controls how local approximation errors 
at individual exercise dates accumulate backward in time. In the 
Longstaff-Schwartz framework, this local error arises from fitting a new 
regression at each step, so the bound involves the regression error at every 
date. In our CME-based approach, the same offline-trained operator is reused 
at every exercise date, so each local error term in the sum is bounded by the 
same quantity, see Equation~\eqref{eq:local_err}.
From Equations~\eqref{eq:init_value}--\eqref{eq:init_value_approx} and the
inequality
$|\max(a,b) - \max(a,c)| \leq |b-c|$, we have
\begin{align*}
\big| \tilde{\mathbb{E}}[P_{\tilde{\tau}_0}]
- \mathbb{E}[P_{\tau_0}] \big|
&\le
\big| \tilde{\mathbb{E}}[P_{\tilde{\tau}_1} \mid S_0]
- \mathbb{E}[P_{\tau_1} \mid S_0] \big|
=
\big\| \tilde{\mathbb{E}}[P_{\tilde{\tau}_1} \mid S_0]
- \mathbb{E}[P_{\tau_1} \mid S_0] \big\|_{L^2_0},
\end{align*}
where the last equality follows from the fact that $S_0$ is deterministic.
Applying Proposition~\ref{prop:local_err_prop} with $k = 0$, we obtain
\begin{align*}
\big\| \tilde{\mathbb{E}}[P_{\tilde{\tau}_1} \mid S_0]
- \mathbb{E}[P_{\tau_1} \mid S_0] \big\|_{L^2_0}
&\le
2 \sum_{j=0}^{n_T - 1}
\big\| \tilde{\mathbb{E}}[P_{\tilde{\tau}_{j+1}} \mid S_j]
- \mathbb{E}[P_{{\tilde \tau}_{j+1}} \mid S_j] \big\|_{L^2_j}.
\end{align*}
Since in the offline CME setting, the local error bound \eqref{eq:local_err} is
the same at every exercise date, we conclude that the pricing error at the
initial time satisfies
\[
\big| \tilde{\mathbb{E}}[P_{\tilde{\tau}_0}]
- \mathbb{E}[P_{\tau_0}] \big|
\le
2 \sum_{j=0}^{n_T-1} \left(\sqrt{\delta^S_0} 
+ \kappa_X  \sqrt{\delta^{LR}}\right)
=
2 n_T \left(\sqrt{\delta^S_0} + \kappa_X \sqrt{\delta^{LR}}\right).
\]
This shows that the total pricing error grows at most linearly in the number 
of exercise dates $n_T$, with the proportionality constant controlled by 
both the statistical error $\delta^S_0$ and the low-rank approximation error 
$\delta^{LR}$ analyzed in Section~\ref{sec:conv_rate}. 
This structure is computationally advantageous: the offline phase is performed 
once, and the online backward recursion reduces to a sequence of matrix-vector 
products. The complete pricing procedure is summarized in 
Algorithm~\ref{alg:off_cmelr}, which will be used in the numerical experiments 
of Section~\ref{sec:num_test}.

\begin{algorithm}[htb]
\caption{Offline low-rank CME for American option pricing}
\label{alg:off_cmelr}
\begin{algorithmic}[1]

\Require Simulated paths $\{S_k^{(i)}\}^{i=1,\dots,n}_{k=0,\dots,n_T}$,
         payoff $\{P_k^{(i)}\}^{i=1,\dots,n}_{k=0,\dots,n_T}$,
         kernels $k_\Xcal, k_\Ycal$, regularization parameter $\lambda$, 
         Cholesky tolerance $\varepsilon$
\Ensure  Estimated American option value $\tilde{{\Vcal}}_0$

\vspace{0.4em}
\textbf{// Offline phase: construct the low-rank CME operator}

\State Consider the state variables $\tilde X= (\tilde x_i)_{i=1}^n,
\tilde Y = (\tilde y_i)_{i=1}^n$ at the reference grid time $n_T-1$:
\[
\tilde x_i = S_{n_T-1}^{(i)},
\qquad
\tilde y_i = S_{n_T}^{(i)}
\]

\State Compute the pivoted Cholesky factorization (Thm.~\ref{thmChol})
with tolerance $\varepsilon$ of the kernel matrices $\bm{K}_{\tilde X}, 
\bm{K}_{\tilde Y} \in \R^{n\times n}$ with entries
$[\bm{K}_{\tilde X}]_{ij} = k_\Xcal(\tilde x_i, \tilde x_j)$
and
$[\bm{K}_{\tilde Y}]_{ij} = k_\Ycal(\tilde y_i, \tilde y_j)$:

\[
\bm{K}_{\tilde X} \approx \bm{L}_{\tilde X} \bm{L}_{\tilde X}^\top, 
\quad \bm{K}_{\tilde Y} \approx \bm{L}_{\tilde Y} \bm{L}_{\tilde Y}^\top,
\]
together with the basis matrices $\bm{B}_{\tilde X} 
\in \R^{n\times m_{\tilde X}},
\bm{B}_{\tilde Y} \in \R^{n\times m_{\tilde Y}}$

\State Compute the spectral decompositions:

\[\bm{V}_{\tilde X} \bm{\Lambda}_{\tilde X} \bm{V}_{\tilde X}^\top 
= \bm{L}_{\tilde X}^\top \bm{L}_{\tilde X}, \qquad \bm{V}_{\tilde Y} 
\bm{\Lambda}_{\tilde Y} \bm{V}_{\tilde Y}^\top 
= \bm{L}_{\tilde Y}^\top \bm{L}_{\tilde Y}\]
and set $\R^{n \times m_{\tilde X}} \ni \bm{Q}_{\tilde X} 
= \bm{B}_{\tilde X} \bm{V}_{\tilde X}, \  
\R^{n \times m_{\tilde Y}} \ni \bm{Q}_{\tilde Y} 
= \bm{B}_{\tilde Y} \bm{V}_{\tilde Y}$

\State Precompute and store $\tilde{\bm F} \in 
\R^{m_{\tilde Y} \times m_{\tilde X}}$ as in \eqref{eq:Ftilde}:
\[
\tilde{\bm F} = (\bm L_{\tilde Y} \bm V_{\tilde Y})^\T (\bm L_{\tilde X} 
\bm V_{\tilde X}) (\bm \Lambda_{\tilde X} + n \lambda \bm I_{m_{\tilde X}})^{-1}
\]

\vspace{0.4em}
\textbf{// Initialization}

\State Set terminal values: $\tilde \Vcal^{(i)}_{n_T} = P_{n_T}^{(i)}$, 
$i = 1,\dots,n$

\vspace{0.4em}
\textbf{// Online phase: backward recursion}
\For{$k = n_T - 1, \dots, 1$}
    \State Set future values $\bm{f}\in \R^n$ with components 
    $f^{(i)} = \tilde \Vcal^{(i)}_{k+1}$, $i = 1,\dots,n$
    \State Consider the current states $ X= ( x_i)_{i=1}^n$:
    \[
     x_i = S_{k}^{(i)}
    \]
\State Evaluate the feature vector 
$\bs\Phi_{\tilde X}(x_i) = [k_\Xcal(\tilde x_j, x_i)]_{j=1}^n$
        
\State Estimate the continuation value $\tilde{C} \in \R^n$ 
via \eqref{eq:sol_lowrank2}:
        \[
        \tilde{C} = \bm{f}^\T \, \bm{Q}_{\tilde Y}  \tilde{\bm F}  
        \bm{Q}_{\tilde X}^\T \, \bs\Phi_{\tilde X}^\top(x)
        \]

    \State Apply the optimal stopping rule as in 
    \eqref{eq:backward_recursion_intro}:
\[
  \tilde{\mathcal{V}}_k^{(i)} = \max\big\{P_k^{(i)},\, \tilde{C}^{(i)}\big\},
  \quad i = 1, \dots, n.
\]
\EndFor
\State \Return $\tilde{{\Vcal}}_0 
= \max \left\{P_0 , \frac{1}{n}\sum_{i=1}^n \tilde
\Vcal_1^{(i)} \right\}$
\end{algorithmic}
\end{algorithm}

\section{Numerical experiments}\label{sec:num_test}
In this section, we present numerical experiments to evaluate the proposed 
(offline) low-rank conditional mean embedding algorithm (CME-LR), in American
option pricing problems. In particular, we use Algorithm \ref{alg:off_cmelr},
for pricing American put options under the Heston stochastic volatility model
\citet{heston93}, whose dynamics are defined through the stochastic differential
equations
\begin{align*}
  \operatorname{d}\!\log S_t &= \Bigl(r - \tfrac{1}{2}\nu_t\Bigr)
  \operatorname{d}\!t 
    + \sqrt{\nu_t}\operatorname{d}\!W_t^S, \\
    \operatorname{d}\!\nu_t &= \kappa (\theta - \nu_t) 
    \operatorname{d}\!t + \xi \sqrt{\nu_t} \operatorname{d}\!W_t^\nu,
\end{align*}
with $\operatorname{corr}( \operatorname{d}\!W_t^S, 
\operatorname{d}\!W_t^\nu) = \rho$.
Here $\kappa$ is the mean-reversion speed, $\theta$ the long-term variance,
$\xi$ the volatility of variance, and $\rho$ the correlation between the asset
and its variance process.
The model parameters used throughout are summarized in
Table~\ref{tab:heston_params}.
\begin{table}[ht]
\centering
\renewcommand{\arraystretch}{1.3}
\begin{tabular}{lll}
\hline
\textbf{Parameter} & \textbf{Symbol} & \textbf{Value} \\
\hline
Initial asset price      & $S_0$    & $100$ \\
Initial variance         & $\nu_0$  & $0.04$\\
Risk-free rate           & $r$      & $0$ \\
Mean-reversion speed     & $\kappa$ & $2$ \\
Long-term variance       & $\theta$ & $0.04$ \\
Volatility of variance   & $\xi$    & $0.3$ \\
Correlation              & $\rho$   & $-0.7$ \\
\hline
\end{tabular}
\caption{Heston model parameters used in the numerical experiments.}
\label{tab:heston_params}
\end{table}
The variance SDE is discretized via a full-truncation Euler-Maruyama scheme,
with the variance truncated from below at $10^{-8}$ to prevent negative values.
The number of time steps is set to $n_T = \max(20,\, \lfloor 52 T \rfloor)$,
corresponding approximately to weekly monitoring.

To apply Algorithm \ref{alg:off_cmelr}, the state variable used in the CME
framework is defined through 
$X = (\log S_k,\, \nu_k)$,  and  $Y = (\log S_{k+1},\, \nu_{k+1})$.
After comparing several kernel combinations, the best-performing configuration
uses the following kernels.
For the state variable $X = (\log S_k, \nu_k)$, we adopt a polynomial kernel of
order $4$,
\[
    k_\Xcal(x, x') = (1 + x^\top x')^4.
\]
\begin{remark}
The convergence result of \citet[Theorem~3]{LiMeunier_etal2024},
used here to bound the statistical error in Equation~\eqref{eq:stat_err_Li},
formally requires $k_\Xcal$ to be uniformly bounded, 
see Appendix~\ref{app:stat_error}. However, this assumption is only invoked in
the proof to ensure Bochner-integrability of a certain random operator, 
and can be replaced by requiring the first four absolute moments of $X$ to be
finite. Under Heston dynamics, the moments of $\log S_t$ are not automatically
finite to all orders: the moment generating function of the log-price exhibits
explosions outside an explicit critical interval, characterized for general
affine stochastic volatility models by
\citet{kellerressel11}, with the corresponding implications for implied
volatility asymptotics at extreme strikes given by \citet{lee04moment}.
The required four-moment condition therefore restricts the admissible Heston
parameter regime, but is satisfied for the configurations considered in our
experiments.
\end{remark}
For the transition variable $Y = (y_i)_{i=1}^n$, with 
$y_i = \log S_{k+1}^{(i)}$, we have found the dependence on $\nu_{k+1}$ to be
negligible, and use a reduced univariate Mat\'{e}rn-$3/2$ kernel, 
\[
    k_\Ycal(y, y') =
    \bigg(1 + \frac{\sqrt{3}|y - y'|}{\ell}\bigg)
    \exp\bigg(-\frac{\sqrt{3}|y - y'|}{\ell}\bigg),
\]
with lengthscale set by the median heuristic:
$
\ell = \operatorname{median}\bigl\{ |y_i - y_j| : 1 \le i < j \le n \bigr\}. 
$ Note that this kernel pertains to a function space norm-equivalent to the 
Sobolev space $H^2(\mathbb R)$ of twice differentiable functions. 
The regularization parameter is set to $\lambda = n^{-1/2}$,
and the Cholesky tolerance is $\varepsilon = 10^{-5}$. This configuration is
compared against the Longstaff-Schwartz baseline
(LS), which uses a degree-$4$ polynomial basis in $(\log S_k, \nu_k)$.

\begin{remark}[Behavior of the error bounds with $\lambda = n^{-1/2}$]
\leavevmode
\\
\textbf{Statistical error.}
The regularization parameter $\lambda = n^{-1/2}$ is a standard data-driven 
choice that does not require knowledge of the smoothness parameters $\beta$ and
$p$ appearing in Theorem~\ref{thm:stat_error}. To understand its theoretical 
implications, note that the optimal $\lambda$ in case~(2) of 
Theorem~\ref{thm:stat_error} is $\lambda_n = \Theta(n^{-1/(\beta+p)})$. 
The choice $\lambda = n^{-1/2}$ therefore corresponds to implicitly assuming 
$\beta + p = 2$. For the polynomial kernel of order~$4$, the decay parameter 
$p$ is not known explicitly, and whether the well-specified ($\beta \geq 1$) 
or misspecified ($\beta < 1$) case applies depends on the regularity of the 
continuation value under the Heston model. Taking the conservative worst-case 
scenario $\gamma = 0$, i.e., the $L^2$ norm, see Appendix~\ref{app:stat_error}, 
and $\beta + p = 2$, the statistical error bound of Theorem~\ref{thm:stat_error} 
gives a rate of $n^{-\beta/2}$, which for $\beta = 1$ reduces to $n^{-1/2}$,
the classical nonparametric rate. In particular, $\delta^S_0 \to 0$ as 
$n \to \infty$ for any fixed $\beta > 0$.

\textbf{Low-rank approximation error.}
With $\lambda = n^{-1/2}$, the coefficient matrix satisfies 
$\|\bm F\|_F^2 = \|(\bm K_X + n\lambda \bm I_n)^{-1}\|_F^2 
= \mathcal{O}\big(1/(n\lambda)^2\big) = \mathcal{O}(1)$, 
while $\operatorname{trace}(\bm K_X) + \operatorname{trace}(\bm K_Y) 
= \mathcal{O}(n)$, 
so the first term in $\delta^{LR}$ is $\mathcal{O}(\varepsilon)$. 
The second term however satisfies $\varepsilon^2/(n\lambda)^4 \cdot 
\operatorname{trace}(\bm K_X)\operatorname{trace}(\bm K_Y) 
= \mathcal{O}(\varepsilon^2)$, 
also constant in $n$. Hence $\delta^{LR}$ does not vanish as $n \to \infty$ 
for fixed $\varepsilon$, in contrast to the statistical error. 
Instead, $\delta^{LR} \to 0$ as $\varepsilon \to 0$, reflecting that 
the low-rank error is controlled by the Cholesky tolerance rather than 
the sample size.
\end{remark}

\definecolor{colorLS} {RGB}{31,119,180}
\definecolor{colorCME}{RGB}{214,39,40}
\definecolor{cNA}{RGB}{55,126,184}
\definecolor{cNB}{RGB}{77,175,74}
\definecolor{cNC}{RGB}{228,26,28}
\definecolor{cND}{RGB}{152,78,163}

\newcounter{cibandcnt}
\newcommand{\ciband}[4]{%
\stepcounter{cibandcnt}%
\edef\pathlo{plo\thecibandcnt}%
\edef\pathi{phi\thecibandcnt}%
\addplot[#1, opacity=0.0, forget plot, name path=\pathlo]
    table[x=N, y=#2, col sep=comma]{#4};
\addplot[#1, opacity=0.0, forget plot, name path=\pathi]
    table[x=N, y=#3, col sep=comma]{#4};
\addplot[#1, fill opacity=0.15, draw=none, forget plot]
    fill between[of=\pathlo and \pathi];
}

\newcommand{\cibandMK}[4]{%
\stepcounter{cibandcnt}%
\edef\pathlo{plo\thecibandcnt}%
\edef\pathi{phi\thecibandcnt}%
\addplot[#1, opacity=0.0, forget plot, name path=\pathlo]
    table[x=logmoneyness, y=#2, col sep=comma]{#4};
\addplot[#1, opacity=0.0, forget plot, name path=\pathi]
    table[x=logmoneyness, y=#3, col sep=comma]{#4};
\addplot[#1, fill opacity=0.15, draw=none, forget plot]
    fill between[of=\pathlo and \pathi];
}


\begin{figure}[htb]
\centering
\begin{tabular}{cc}
\begin{tikzpicture}
\begin{axis}[title={$T=1/12$}, xlabel={$n$}, ylabel={$\log_{10}(\text{time} / \mu s)$},
    xmode=log, xmin=100, xmax=100000,
    xtick={100,1000,10000,100000}, xticklabels={$10^2$,$10^3$,$10^4$,$10^5$},
    minor xtick={}, minor ytick={}, ymin=1.5, ymax=7,
    width=6cm, height=5cm, grid=both,
    tick label style={font=\small}, label style={font=\small},
    title style={font=\small\bfseries}]
\addplot[colorLS, opacity=0.0, forget plot, name path=pathlo1]
    table[x=N, y=lo_logtime_poly, col sep=comma]{IMG_apr26/winner_time_T1.csv};
\addplot[colorLS, opacity=0.0, forget plot, name path=pathhi1]
    table[x=N, y=hi_logtime_poly, col sep=comma]{IMG_apr26/winner_time_T1.csv};
\addplot[colorLS, fill opacity=0.15, draw=none, forget plot]
    fill between[of=pathlo1 and pathhi1];
\addplot[colorCME, opacity=0.0, forget plot, name path=pathlo2]
    table[x=N, y=lo_logtime_cme, col sep=comma]{IMG_apr26/winner_time_T1.csv};
\addplot[colorCME, opacity=0.0, forget plot, name path=pathhi2]
    table[x=N, y=hi_logtime_cme, col sep=comma]{IMG_apr26/winner_time_T1.csv};
\addplot[colorCME, fill opacity=0.15, draw=none, forget plot]
    fill between[of=pathlo2 and pathhi2];
\addplot[colorLS, mark=*, mark size=2pt, line width=1.2pt]
    table[x=N, y=mean_logtime_poly, col sep=comma]{IMG_apr26/winner_time_T1.csv};
\addplot[colorCME, mark=square*, mark size=2pt, line width=1.2pt, dashed]
    table[x=N, y=mean_logtime_cme, col sep=comma]{IMG_apr26/winner_time_T1.csv};
\end{axis}
\end{tikzpicture}
&
\begin{tikzpicture}
\begin{axis}[title={$T=1/2$}, xlabel={$n$},
    xmode=log, xmin=100, xmax=100000,
    xtick={100,1000,10000,100000}, xticklabels={$10^2$,$10^3$,$10^4$,$10^5$},
    minor xtick={}, minor ytick={}, ymin=1.5, ymax=7,
    width=6cm, height=5cm, grid=both,
    tick label style={font=\small}, label style={font=\small},
    title style={font=\small\bfseries}]
\addplot[colorLS, opacity=0.0, forget plot, name path=pathlo3]
    table[x=N, y=lo_logtime_poly, col sep=comma]{IMG_apr26/winner_time_T2.csv};
\addplot[colorLS, opacity=0.0, forget plot, name path=pathhi3]
    table[x=N, y=hi_logtime_poly, col sep=comma]{IMG_apr26/winner_time_T2.csv};
\addplot[colorLS, fill opacity=0.15, draw=none, forget plot]
    fill between[of=pathlo3 and pathhi3];
\addplot[colorCME, opacity=0.0, forget plot, name path=pathlo4]
    table[x=N, y=lo_logtime_cme, col sep=comma]{IMG_apr26/winner_time_T2.csv};
\addplot[colorCME, opacity=0.0, forget plot, name path=pathhi4]
    table[x=N, y=hi_logtime_cme, col sep=comma]{IMG_apr26/winner_time_T2.csv};
\addplot[colorCME, fill opacity=0.15, draw=none, forget plot]
    fill between[of=pathlo4 and pathhi4];
\addplot[colorLS, mark=*, mark size=2pt, line width=1.2pt]
    table[x=N, y=mean_logtime_poly, col sep=comma]{IMG_apr26/winner_time_T2.csv};
\addplot[colorCME, mark=square*, mark size=2pt, line width=1.2pt, dashed]
    table[x=N, y=mean_logtime_cme, col sep=comma]{IMG_apr26/winner_time_T2.csv};
\end{axis}
\end{tikzpicture}
\\[6pt]
\begin{tikzpicture}
\begin{axis}[title={$T=1$}, xlabel={$n$}, ylabel={$\log_{10}(\text{time} / \mu s)$},
    xmode=log, xmin=100, xmax=100000,
    xtick={100,1000,10000,100000}, xticklabels={$10^2$,$10^3$,$10^4$,$10^5$},
    minor xtick={}, minor ytick={}, ymin=1.5, ymax=7,
    width=6cm, height=5cm, grid=both,
    tick label style={font=\small}, label style={font=\small},
    title style={font=\small\bfseries}]
\addplot[colorLS, opacity=0.0, forget plot, name path=pathlo5]
    table[x=N, y=lo_logtime_poly, col sep=comma]{IMG_apr26/winner_time_T3.csv};
\addplot[colorLS, opacity=0.0, forget plot, name path=pathhi5]
    table[x=N, y=hi_logtime_poly, col sep=comma]{IMG_apr26/winner_time_T3.csv};
\addplot[colorLS, fill opacity=0.15, draw=none, forget plot]
    fill between[of=pathlo5 and pathhi5];
\addplot[colorCME, opacity=0.0, forget plot, name path=pathlo6]
    table[x=N, y=lo_logtime_cme, col sep=comma]{IMG_apr26/winner_time_T3.csv};
\addplot[colorCME, opacity=0.0, forget plot, name path=pathhi6]
    table[x=N, y=hi_logtime_cme, col sep=comma]{IMG_apr26/winner_time_T3.csv};
\addplot[colorCME, fill opacity=0.15, draw=none, forget plot]
    fill between[of=pathlo6 and pathhi6];
\addplot[colorLS, mark=*, mark size=2pt, line width=1.2pt]
    table[x=N, y=mean_logtime_poly, col sep=comma]{IMG_apr26/winner_time_T3.csv};
\addplot[colorCME, mark=square*, mark size=2pt, line width=1.2pt, dashed]
    table[x=N, y=mean_logtime_cme, col sep=comma]{IMG_apr26/winner_time_T3.csv};
\end{axis}
\end{tikzpicture}
&
\begin{tikzpicture}
\begin{axis}[title={$T=2$}, xlabel={$n$},
    xmode=log, xmin=100, xmax=100000,
    xtick={100,1000,10000,100000}, xticklabels={$10^2$,$10^3$,$10^4$,$10^5$},
    minor xtick={}, minor ytick={}, ymin=1.5, ymax=7,
    width=6cm, height=5cm, grid=both,
    tick label style={font=\small}, label style={font=\small},
    title style={font=\small\bfseries}]
\addplot[colorLS, opacity=0.0, forget plot, name path=pathlo7]
    table[x=N, y=lo_logtime_poly, col sep=comma]{IMG_apr26/winner_time_T4.csv};
\addplot[colorLS, opacity=0.0, forget plot, name path=pathhi7]
    table[x=N, y=hi_logtime_poly, col sep=comma]{IMG_apr26/winner_time_T4.csv};
\addplot[colorLS, fill opacity=0.15, draw=none, forget plot]
    fill between[of=pathlo7 and pathhi7];
\addplot[colorCME, opacity=0.0, forget plot, name path=pathlo8]
    table[x=N, y=lo_logtime_cme, col sep=comma]{IMG_apr26/winner_time_T4.csv};
\addplot[colorCME, opacity=0.0, forget plot, name path=pathhi8]
    table[x=N, y=hi_logtime_cme, col sep=comma]{IMG_apr26/winner_time_T4.csv};
\addplot[colorCME, fill opacity=0.15, draw=none, forget plot]
    fill between[of=pathlo8 and pathhi8];
\addplot[colorLS, mark=*, mark size=2pt, line width=1.2pt]
    table[x=N, y=mean_logtime_poly, col sep=comma]{IMG_apr26/winner_time_T4.csv};
\addplot[colorCME, mark=square*, mark size=2pt, line width=1.2pt, dashed]
    table[x=N, y=mean_logtime_cme, col sep=comma]{IMG_apr26/winner_time_T4.csv};
\end{axis}
\end{tikzpicture}
\\[6pt]
\multicolumn{2}{c}{%
\begin{tikzpicture}
\draw[colorLS, line width=1.2pt] (0,0) -- (0.6,0);
\fill[colorLS] (0.3,0) circle (2pt);
\node[right] at (0.6,0) {\small LS};
\draw[colorCME, line width=1.2pt, dashed] (3.2,0) -- (3.8,0);
\fill[colorCME] (3.5,0) circle (2pt);
\node[right] at (3.8,0) {\small CME-LR};
\end{tikzpicture}}
\end{tabular}
\caption{Mean $\log_{10}$ computation time (in $\mu$s) as a function of $n$,
for LS (blue solid) and CME-LR (red dashed).
}
\label{fig:timing}
\end{figure}
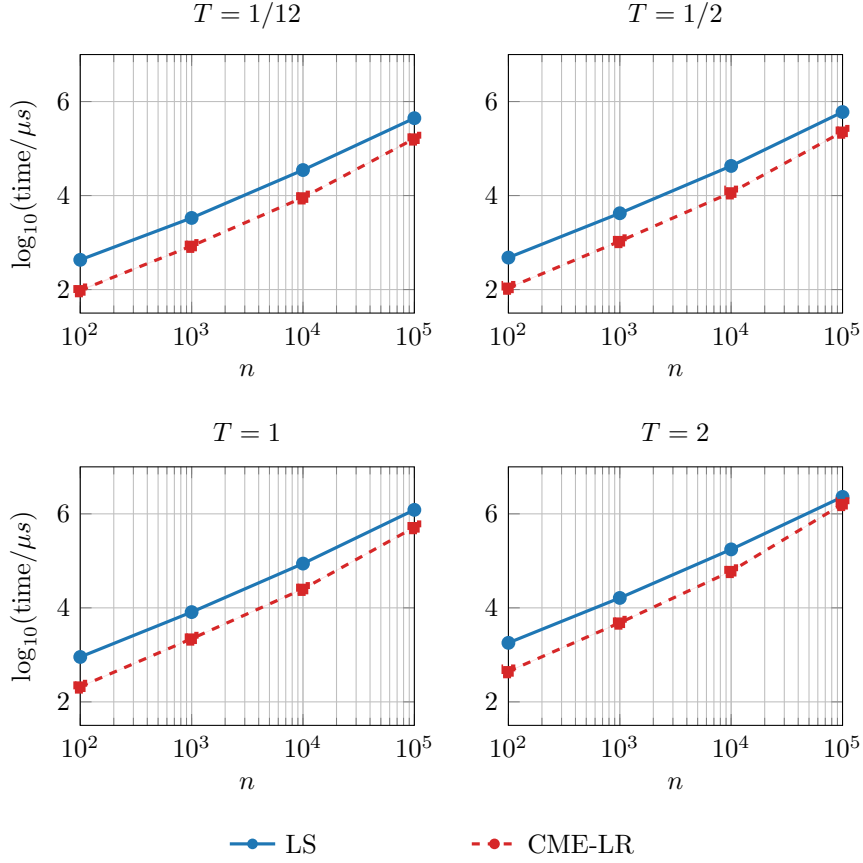


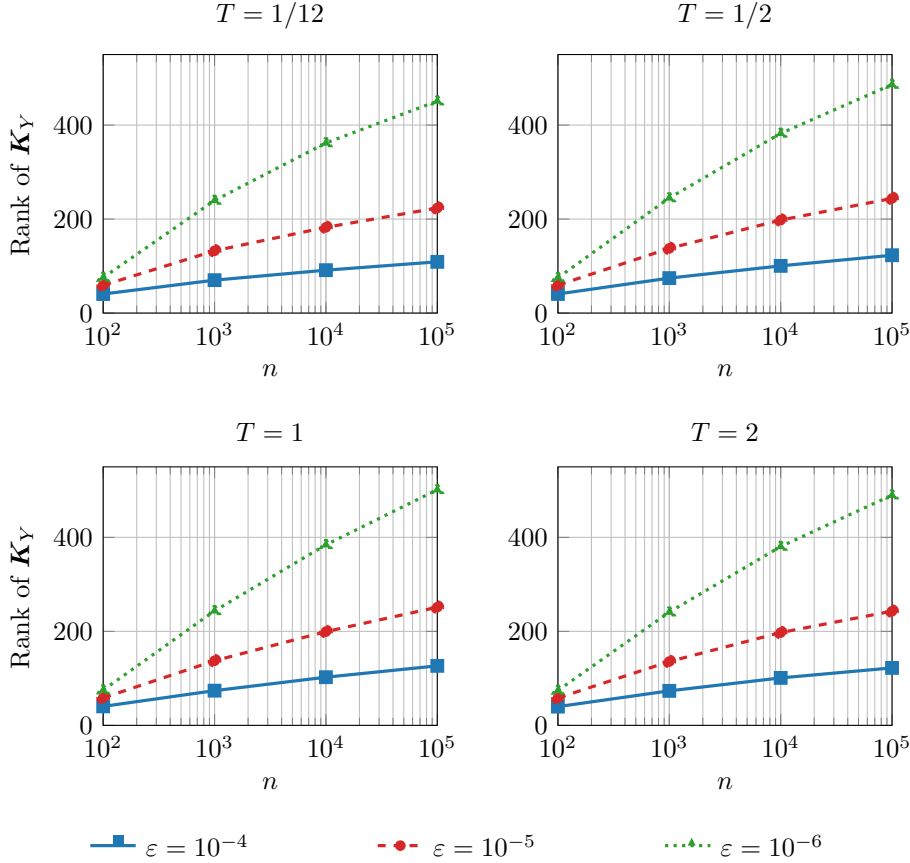
\begin{figure}[htb]
\centering

\definecolor{cEps4}{RGB}{31,119,180}
\definecolor{cEps5}{RGB}{214,39,40}
\definecolor{cEps6}{RGB}{44,160,44}

\begin{tabular}{cc}
\begin{tikzpicture}
\begin{axis}[title={$T=1/12$}, xlabel={$n$}, ylabel={Rank of $\bm{K}_Y$},
    xmode=log, xmin=100, xmax=100000,
    xtick={100,1000,10000,100000}, xticklabels={$10^2$,$10^3$,$10^4$,$10^5$},
    minor xtick={}, minor ytick={}, ymin=0, ymax=550,
    width=6cm, height=5cm, grid=both,
    tick label style={font=\small}, label style={font=\small},
    title style={font=\small\bfseries}]
\addplot[cEps4, mark=square*, mark size=2pt, line width=1.2pt]
    coordinates {(100,40.57)(1000,70.04)(10000,91.28)(100000,108.99)};
\addplot[cEps5, mark=*, mark size=2pt, line width=1.2pt, dashed]
    coordinates {(100,59.23)(1000,132.48)(10000,182.46)(100000,223.48)};
\addplot[cEps6, mark=triangle*, mark size=2pt, line width=1.2pt, dotted]
    coordinates {(100,75.60)(1000,239.50)(10000,361.89)(100000,450.93)};
\end{axis}
\end{tikzpicture}
&
\begin{tikzpicture}
\begin{axis}[title={$T=1/2$}, xlabel={$n$},
    xmode=log, xmin=100, xmax=100000,
    xtick={100,1000,10000,100000}, xticklabels={$10^2$,$10^3$,$10^4$,$10^5$},
    minor xtick={}, minor ytick={}, ymin=0, ymax=550,
    width=6cm, height=5cm, grid=both,
    tick label style={font=\small}, label style={font=\small},
    title style={font=\small\bfseries}]
\addplot[cEps4, mark=square*, mark size=2pt, line width=1.2pt]
    coordinates {(100,40.62)(1000,74.40)(10000,100.45)(100000,123.02)};
\addplot[cEps5, mark=*, mark size=2pt, line width=1.2pt, dashed]
    coordinates {(100,59.34)(1000,137.69)(10000,197.90)(100000,243.92)};
\addplot[cEps6, mark=triangle*, mark size=2pt, line width=1.2pt, dotted]
    coordinates {(100,75.22)(1000,244.86)(10000,382.27)(100000,486.20)};
\end{axis}
\end{tikzpicture}
\\[6pt]
\begin{tikzpicture}
\begin{axis}[title={$T=1$}, xlabel={$n$}, ylabel={Rank of $\bm{K}_Y$},
    xmode=log, xmin=100, xmax=100000,
    xtick={100,1000,10000,100000}, xticklabels={$10^2$,$10^3$,$10^4$,$10^5$},
    minor xtick={}, minor ytick={}, ymin=0, ymax=550,
    width=6cm, height=5cm, grid=both,
    tick label style={font=\small}, label style={font=\small},
    title style={font=\small\bfseries}]
\addplot[cEps4, mark=square*, mark size=2pt, line width=1.2pt]
    coordinates {(100,40.39)(1000,73.87)(10000,102.47)(100000,126.63)};
\addplot[cEps5, mark=*, mark size=2pt, line width=1.2pt, dashed]
    coordinates {(100,58.51)(1000,137.45)(10000,199.47)(100000,251.39)};
\addplot[cEps6, mark=triangle*, mark size=2pt, line width=1.2pt, dotted]
    coordinates {(100,74.98)(1000,243.62)(10000,384.19)(100000,501.28)};
\end{axis}
\end{tikzpicture}
&
\begin{tikzpicture}
\begin{axis}[title={$T=2$}, xlabel={$n$},
    xmode=log, xmin=100, xmax=100000,
    xtick={100,1000,10000,100000}, xticklabels={$10^2$,$10^3$,$10^4$,$10^5$},
    minor xtick={}, minor ytick={}, ymin=0, ymax=550,
    width=6cm, height=5cm, grid=both,
    tick label style={font=\small}, label style={font=\small},
    title style={font=\small\bfseries}]
\addplot[cEps4, mark=square*, mark size=2pt, line width=1.2pt]
    coordinates {(100,39.86)(1000,73.43)(10000,100.97)(100000,122.28)};
\addplot[cEps5, mark=*, mark size=2pt, line width=1.2pt, dashed]
    coordinates {(100,58.21)(1000,135.81)(10000,197.42)(100000,242.93)};
\addplot[cEps6, mark=triangle*, mark size=2pt, line width=1.2pt, dotted]
    coordinates {(100,74.21)(1000,240.82)(10000,380.43)(100000,489.76)};
\end{axis}
\end{tikzpicture}
\\[6pt]
\multicolumn{2}{c}{%
\begin{tikzpicture}
\draw[cEps4, line width=1.2pt] (0,0)--(0.6,0);
\fill[cEps4] (0.3,0) rectangle ++(4pt,4pt) +(-2pt,-2pt);
\node[right] at (0.6,0) {\small $\varepsilon=10^{-4}$};
\draw[cEps5, line width=1.2pt, dashed] (3.8,0)--(4.4,0);
\fill[cEps5] (4.1,0) circle (2pt);
\node[right] at (4.4,0) {\small $\varepsilon=10^{-5}$};
\draw[cEps6, line width=1.2pt, dotted] (7.6,0)--(8.2,0);
\fill[cEps6] (7.9,0) -- ++(3pt,0) -- ++(-1.5pt,3pt) -- cycle;
\node[right] at (8.2,0) {\small $\varepsilon=10^{-6}$};
\end{tikzpicture}}
\end{tabular}
\caption{Mean rank of the kernel matrix $\bm{K}_Y$ selected by the pivoted Cholesky
decomposition as a function of $n$, across four maturities, for three values of the
Cholesky tolerance $\varepsilon \in \{10^{-4}, 10^{-5}, 10^{-6}\}$.}
\label{fig:rank}
\end{figure}

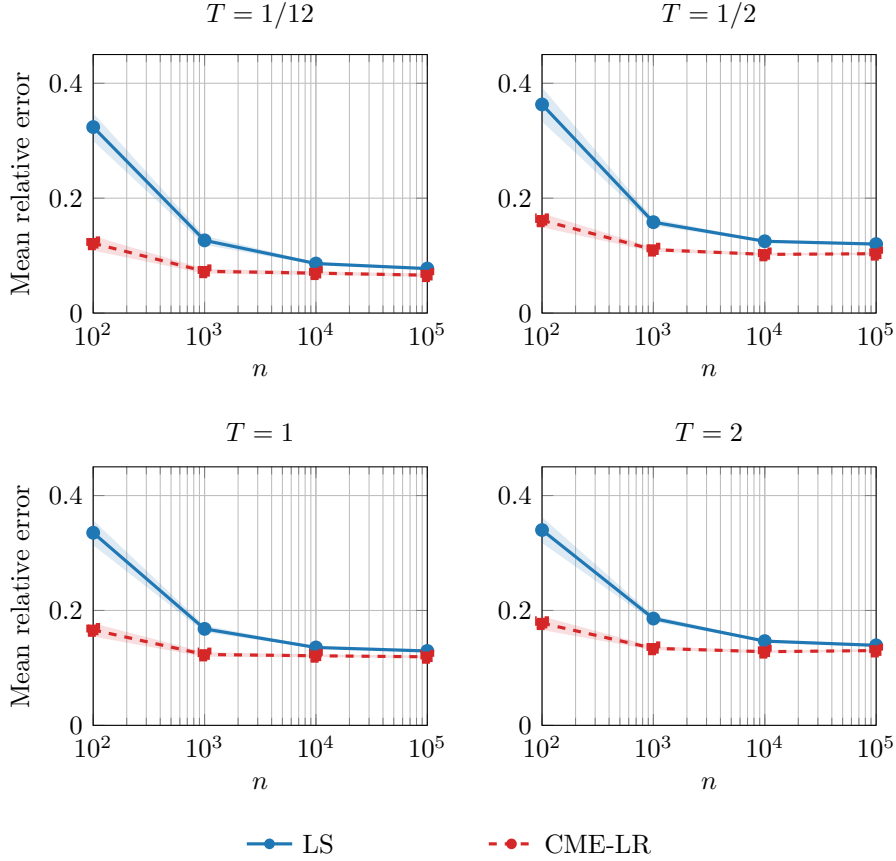
\begin{figure}[ht]
\centering
\begin{tabular}{cc}
\begin{tikzpicture}
\begin{axis}[title={$T=1/12$}, xlabel={$n$}, ylabel={Mean relative error},
    xmode=log, xmin=100, xmax=100000,
    xtick={100,1000,10000,100000}, xticklabels={$10^2$,$10^3$,$10^4$,$10^5$},
    minor xtick={}, minor ytick={}, ymin=0, ymax=0.45,
    width=6cm, height=5cm, grid=both,
    legend pos=north east, legend style={font=\small,draw=none},
    tick label style={font=\small}, label style={font=\small},
    title style={font=\small\bfseries}]
\addplot[colorLS, opacity=0.0, forget plot, name path=pathlo9]
    table[x=N, y=lo_poly, col sep=comma]{IMG_apr26/winner_err_T1.csv};
\addplot[colorLS, opacity=0.0, forget plot, name path=pathhi9]
    table[x=N, y=hi_poly, col sep=comma]{IMG_apr26/winner_err_T1.csv};
\addplot[colorLS, fill opacity=0.15, draw=none, forget plot]
    fill between[of=pathlo9 and pathhi9];
\addplot[colorCME, opacity=0.0, forget plot, name path=pathlo10]
    table[x=N, y=lo_cme, col sep=comma]{IMG_apr26/winner_err_T1.csv};
\addplot[colorCME, opacity=0.0, forget plot, name path=pathhi10]
    table[x=N, y=hi_cme, col sep=comma]{IMG_apr26/winner_err_T1.csv};
\addplot[colorCME, fill opacity=0.15, draw=none, forget plot]
    fill between[of=pathlo10 and pathhi10];
\addplot[colorLS,  mark=*, mark size=2pt, line width=1.2pt]
    table[x=N, y=mean_relerr_poly, col sep=comma]{IMG_apr26/winner_err_T1.csv};
\addplot[colorCME, mark=square*, mark size=2pt, line width=1.2pt, dashed]
    table[x=N, y=mean_relerr_cme,  col sep=comma]{IMG_apr26/winner_err_T1.csv};
\end{axis}
\end{tikzpicture}
&
\begin{tikzpicture}
\begin{axis}[title={$T=1/2$}, xlabel={$n$},
    xmode=log, xmin=100, xmax=100000,
    xtick={100,1000,10000,100000}, xticklabels={$10^2$,$10^3$,$10^4$,$10^5$},
    minor xtick={}, minor ytick={}, ymin=0, ymax=0.45,
    width=6cm, height=5cm, grid=both,
    legend pos=north east, legend style={font=\small,draw=none},
    tick label style={font=\small}, label style={font=\small},
    title style={font=\small\bfseries}]
\addplot[colorLS, opacity=0.0, forget plot, name path=pathlo11]
    table[x=N, y=lo_poly, col sep=comma]{IMG_apr26/winner_err_T2.csv};
\addplot[colorLS, opacity=0.0, forget plot, name path=pathhi11]
    table[x=N, y=hi_poly, col sep=comma]{IMG_apr26/winner_err_T2.csv};
\addplot[colorLS, fill opacity=0.15, draw=none, forget plot]
    fill between[of=pathlo11 and pathhi11];
\addplot[colorCME, opacity=0.0, forget plot, name path=pathlo12]
    table[x=N, y=lo_cme, col sep=comma]{IMG_apr26/winner_err_T2.csv};
\addplot[colorCME, opacity=0.0, forget plot, name path=pathhi12]
    table[x=N, y=hi_cme, col sep=comma]{IMG_apr26/winner_err_T2.csv};
\addplot[colorCME, fill opacity=0.15, draw=none, forget plot]
    fill between[of=pathlo12 and pathhi12];
\addplot[colorLS,  mark=*, mark size=2pt, line width=1.2pt]
    table[x=N, y=mean_relerr_poly, col sep=comma]{IMG_apr26/winner_err_T2.csv};
\addplot[colorCME, mark=square*, mark size=2pt, line width=1.2pt, dashed]
    table[x=N, y=mean_relerr_cme,  col sep=comma]{IMG_apr26/winner_err_T2.csv};
\end{axis}
\end{tikzpicture}
\\[6pt]
\begin{tikzpicture}
\begin{axis}[title={$T=1$}, xlabel={$n$}, ylabel={Mean relative error},
    xmode=log, xmin=100, xmax=100000,
    xtick={100,1000,10000,100000}, xticklabels={$10^2$,$10^3$,$10^4$,$10^5$},
    minor xtick={}, minor ytick={}, ymin=0, ymax=0.45,
    width=6cm, height=5cm, grid=both,
    legend pos=north east, legend style={font=\small,draw=none},
    tick label style={font=\small}, label style={font=\small},
    title style={font=\small\bfseries}]
\addplot[colorLS, opacity=0.0, forget plot, name path=pathlo13]
    table[x=N, y=lo_poly, col sep=comma]{IMG_apr26/winner_err_T3.csv};
\addplot[colorLS, opacity=0.0, forget plot, name path=pathhi13]
    table[x=N, y=hi_poly, col sep=comma]{IMG_apr26/winner_err_T3.csv};
\addplot[colorLS, fill opacity=0.15, draw=none, forget plot]
    fill between[of=pathlo13 and pathhi13];
\addplot[colorCME, opacity=0.0, forget plot, name path=pathlo14]
    table[x=N, y=lo_cme, col sep=comma]{IMG_apr26/winner_err_T3.csv};
\addplot[colorCME, opacity=0.0, forget plot, name path=pathhi14]
    table[x=N, y=hi_cme, col sep=comma]{IMG_apr26/winner_err_T3.csv};
\addplot[colorCME, fill opacity=0.15, draw=none, forget plot]
    fill between[of=pathlo14 and pathhi14];
\addplot[colorLS,  mark=*, mark size=2pt, line width=1.2pt]
    table[x=N, y=mean_relerr_poly, col sep=comma]{IMG_apr26/winner_err_T3.csv};
\addplot[colorCME, mark=square*, mark size=2pt, line width=1.2pt, dashed]
    table[x=N, y=mean_relerr_cme,  col sep=comma]{IMG_apr26/winner_err_T3.csv};
\end{axis}
\end{tikzpicture}
&
\begin{tikzpicture}
\begin{axis}[title={$T=2$}, xlabel={$n$},
    xmode=log, xmin=100, xmax=100000,
    xtick={100,1000,10000,100000}, xticklabels={$10^2$,$10^3$,$10^4$,$10^5$},
    minor xtick={}, minor ytick={}, ymin=0, ymax=0.45,
    width=6cm, height=5cm, grid=both,
    legend pos=north east, legend style={font=\small,draw=none},
    tick label style={font=\small}, label style={font=\small},
    title style={font=\small\bfseries}]
\addplot[colorLS, opacity=0.0, forget plot, name path=pathlo15]
    table[x=N, y=lo_poly, col sep=comma]{IMG_apr26/winner_err_T4.csv};
\addplot[colorLS, opacity=0.0, forget plot, name path=pathhi15]
    table[x=N, y=hi_poly, col sep=comma]{IMG_apr26/winner_err_T4.csv};
\addplot[colorLS, fill opacity=0.15, draw=none, forget plot]
    fill between[of=pathlo15 and pathhi15];
\addplot[colorCME, opacity=0.0, forget plot, name path=pathlo16]
    table[x=N, y=lo_cme, col sep=comma]{IMG_apr26/winner_err_T4.csv};
\addplot[colorCME, opacity=0.0, forget plot, name path=pathhi16]
    table[x=N, y=hi_cme, col sep=comma]{IMG_apr26/winner_err_T4.csv};
\addplot[colorCME, fill opacity=0.15, draw=none, forget plot]
    fill between[of=pathlo16 and pathhi16];
\addplot[colorLS,  mark=*, mark size=2pt, line width=1.2pt]
    table[x=N, y=mean_relerr_poly, col sep=comma]{IMG_apr26/winner_err_T4.csv};
\addplot[colorCME, mark=square*, mark size=2pt, line width=1.2pt, dashed]
    table[x=N, y=mean_relerr_cme,  col sep=comma]{IMG_apr26/winner_err_T4.csv};
\end{axis}
\end{tikzpicture}
\\[6pt]
\multicolumn{2}{c}{%
\begin{tikzpicture}
\draw[colorLS, line width=1.2pt] (0,0) -- (0.6,0);
\fill[colorLS] (0.3,0) circle (2pt);
\node[right] at (0.6,0) {\small LS};
\draw[colorCME, line width=1.2pt, dashed] (3.2,0) -- (3.8,0);
\fill[colorCME] (3.5,0) circle (2pt);
\node[right] at (3.8,0) {\small CME-LR};
\end{tikzpicture}}
\end{tabular}
\caption{Mean relative implied volatility error as a function of $n$, for LS (blue solid) and CME-LR (red dashed). Shaded bands: 95\% confidence intervals over 100 replications.}
\label{fig:error_agg}
\end{figure}


\begin{figure}[p]
\centering
\begin{tabular}{cc}
\begin{tikzpicture}
\begin{axis}[title={LS, $T=1/12$}, xlabel={$\log(K/S_0)$}, ylabel={Mean relative error},
    ymin=0, ymax=0.8, minor ytick={}, minor xtick={},
    width=6cm, height=5cm, grid=both,
    tick label style={font=\small}, label style={font=\small}, title style={font=\small\bfseries}]
\addplot[cNA, opacity=0.0, forget plot, name path=pathlo17]
    table[x=logmoneyness, y=lo_poly, col sep=comma]{IMG_apr26/error_mk_winner_T1_N1.csv};
\addplot[cNA, opacity=0.0, forget plot, name path=pathhi17]
    table[x=logmoneyness, y=hi_poly, col sep=comma]{IMG_apr26/error_mk_winner_T1_N1.csv};
\addplot[cNA, fill opacity=0.15, draw=none, forget plot]
    fill between[of=pathlo17 and pathhi17];
\addplot[cNB, opacity=0.0, forget plot, name path=pathlo18]
    table[x=logmoneyness, y=lo_poly, col sep=comma]{IMG_apr26/error_mk_winner_T1_N2.csv};
\addplot[cNB, opacity=0.0, forget plot, name path=pathhi18]
    table[x=logmoneyness, y=hi_poly, col sep=comma]{IMG_apr26/error_mk_winner_T1_N2.csv};
\addplot[cNB, fill opacity=0.15, draw=none, forget plot]
    fill between[of=pathlo18 and pathhi18];
\addplot[cNC, opacity=0.0, forget plot, name path=pathlo19]
    table[x=logmoneyness, y=lo_poly, col sep=comma]{IMG_apr26/error_mk_winner_T1_N3.csv};
\addplot[cNC, opacity=0.0, forget plot, name path=pathhi19]
    table[x=logmoneyness, y=hi_poly, col sep=comma]{IMG_apr26/error_mk_winner_T1_N3.csv};
\addplot[cNC, fill opacity=0.15, draw=none, forget plot]
    fill between[of=pathlo19 and pathhi19];
\addplot[cND, opacity=0.0, forget plot, name path=pathlo20]
    table[x=logmoneyness, y=lo_poly, col sep=comma]{IMG_apr26/error_mk_winner_T1_N4.csv};
\addplot[cND, opacity=0.0, forget plot, name path=pathhi20]
    table[x=logmoneyness, y=hi_poly, col sep=comma]{IMG_apr26/error_mk_winner_T1_N4.csv};
\addplot[cND, fill opacity=0.15, draw=none, forget plot]
    fill between[of=pathlo20 and pathhi20];
\addplot[cNA, mark=*,        mark size=1.5pt, line width=1pt]
    table[x=logmoneyness, y=mean_relerr_poly, col sep=comma]{IMG_apr26/error_mk_winner_T1_N1.csv};
\addplot[cNB, mark=square*,  mark size=1.5pt, line width=1pt]
    table[x=logmoneyness, y=mean_relerr_poly, col sep=comma]{IMG_apr26/error_mk_winner_T1_N2.csv};
\addplot[cNC, mark=triangle*, mark size=1.5pt, line width=1pt]
    table[x=logmoneyness, y=mean_relerr_poly, col sep=comma]{IMG_apr26/error_mk_winner_T1_N3.csv};
\addplot[cND, mark=diamond*,  mark size=1.5pt, line width=1pt]
    table[x=logmoneyness, y=mean_relerr_poly, col sep=comma]{IMG_apr26/error_mk_winner_T1_N4.csv};
\end{axis}
\end{tikzpicture}
&\hspace{0.35cm}
\begin{tikzpicture}
\begin{axis}[title={CME-LR, $T=1/12$}, xlabel={$\log(K/S_0)$},
    ymin=0, ymax=0.8, minor ytick={}, minor xtick={},
    width=6cm, height=5cm, grid=both,
    tick label style={font=\small}, label style={font=\small}, title style={font=\small\bfseries}]
\addplot[cNA, opacity=0.0, forget plot, name path=pathlo21]
    table[x=logmoneyness, y=lo_cme, col sep=comma]{IMG_apr26/error_mk_winner_T1_N1.csv};
\addplot[cNA, opacity=0.0, forget plot, name path=pathhi21]
    table[x=logmoneyness, y=hi_cme, col sep=comma]{IMG_apr26/error_mk_winner_T1_N1.csv};
\addplot[cNA, fill opacity=0.15, draw=none, forget plot]
    fill between[of=pathlo21 and pathhi21];
\addplot[cNB, opacity=0.0, forget plot, name path=pathlo22]
    table[x=logmoneyness, y=lo_cme, col sep=comma]{IMG_apr26/error_mk_winner_T1_N2.csv};
\addplot[cNB, opacity=0.0, forget plot, name path=pathhi22]
    table[x=logmoneyness, y=hi_cme, col sep=comma]{IMG_apr26/error_mk_winner_T1_N2.csv};
\addplot[cNB, fill opacity=0.15, draw=none, forget plot]
    fill between[of=pathlo22 and pathhi22];
\addplot[cNC, opacity=0.0, forget plot, name path=pathlo23]
    table[x=logmoneyness, y=lo_cme, col sep=comma]{IMG_apr26/error_mk_winner_T1_N3.csv};
\addplot[cNC, opacity=0.0, forget plot, name path=pathhi23]
    table[x=logmoneyness, y=hi_cme, col sep=comma]{IMG_apr26/error_mk_winner_T1_N3.csv};
\addplot[cNC, fill opacity=0.15, draw=none, forget plot]
    fill between[of=pathlo23 and pathhi23];
\addplot[cND, opacity=0.0, forget plot, name path=pathlo24]
    table[x=logmoneyness, y=lo_cme, col sep=comma]{IMG_apr26/error_mk_winner_T1_N4.csv};
\addplot[cND, opacity=0.0, forget plot, name path=pathhi24]
    table[x=logmoneyness, y=hi_cme, col sep=comma]{IMG_apr26/error_mk_winner_T1_N4.csv};
\addplot[cND, fill opacity=0.15, draw=none, forget plot]
    fill between[of=pathlo24 and pathhi24];
\addplot[cNA, mark=*,        mark size=1.5pt, line width=1pt, dashed]
    table[x=logmoneyness, y=mean_relerr_cme, col sep=comma]{IMG_apr26/error_mk_winner_T1_N1.csv};
\addplot[cNB, mark=square*,  mark size=1.5pt, line width=1pt, dashed]
    table[x=logmoneyness, y=mean_relerr_cme, col sep=comma]{IMG_apr26/error_mk_winner_T1_N2.csv};
\addplot[cNC, mark=triangle*, mark size=1.5pt, line width=1pt, dashed]
    table[x=logmoneyness, y=mean_relerr_cme, col sep=comma]{IMG_apr26/error_mk_winner_T1_N3.csv};
\addplot[cND, mark=diamond*,  mark size=1.5pt, line width=1pt, dashed]
    table[x=logmoneyness, y=mean_relerr_cme, col sep=comma]{IMG_apr26/error_mk_winner_T1_N4.csv};
\end{axis}
\end{tikzpicture}
\\[4pt]
\begin{tikzpicture}
\begin{axis}[title={LS, $T=1/2$}, xlabel={$\log(K/S_0)$}, ylabel={Mean relative error},
    ymin=0, ymax=0.8, minor ytick={}, minor xtick={},
    width=6cm, height=5cm, grid=both,
    tick label style={font=\small}, label style={font=\small}, title style={font=\small\bfseries}]
\addplot[cNA, opacity=0.0, forget plot, name path=pathlo25]
    table[x=logmoneyness, y=lo_poly, col sep=comma]{IMG_apr26/error_mk_winner_T2_N1.csv};
\addplot[cNA, opacity=0.0, forget plot, name path=pathhi25]
    table[x=logmoneyness, y=hi_poly, col sep=comma]{IMG_apr26/error_mk_winner_T2_N1.csv};
\addplot[cNA, fill opacity=0.15, draw=none, forget plot]
    fill between[of=pathlo25 and pathhi25];
\addplot[cNB, opacity=0.0, forget plot, name path=pathlo26]
    table[x=logmoneyness, y=lo_poly, col sep=comma]{IMG_apr26/error_mk_winner_T2_N2.csv};
\addplot[cNB, opacity=0.0, forget plot, name path=pathhi26]
    table[x=logmoneyness, y=hi_poly, col sep=comma]{IMG_apr26/error_mk_winner_T2_N2.csv};
\addplot[cNB, fill opacity=0.15, draw=none, forget plot]
    fill between[of=pathlo26 and pathhi26];
\addplot[cNC, opacity=0.0, forget plot, name path=pathlo27]
    table[x=logmoneyness, y=lo_poly, col sep=comma]{IMG_apr26/error_mk_winner_T2_N3.csv};
\addplot[cNC, opacity=0.0, forget plot, name path=pathhi27]
    table[x=logmoneyness, y=hi_poly, col sep=comma]{IMG_apr26/error_mk_winner_T2_N3.csv};
\addplot[cNC, fill opacity=0.15, draw=none, forget plot]
    fill between[of=pathlo27 and pathhi27];
\addplot[cND, opacity=0.0, forget plot, name path=pathlo28]
    table[x=logmoneyness, y=lo_poly, col sep=comma]{IMG_apr26/error_mk_winner_T2_N4.csv};
\addplot[cND, opacity=0.0, forget plot, name path=pathhi28]
    table[x=logmoneyness, y=hi_poly, col sep=comma]{IMG_apr26/error_mk_winner_T2_N4.csv};
\addplot[cND, fill opacity=0.15, draw=none, forget plot]
    fill between[of=pathlo28 and pathhi28];
\addplot[cNA, mark=*,        mark size=1.5pt, line width=1pt]
    table[x=logmoneyness, y=mean_relerr_poly, col sep=comma]{IMG_apr26/error_mk_winner_T2_N1.csv};
\addplot[cNB, mark=square*,  mark size=1.5pt, line width=1pt]
    table[x=logmoneyness, y=mean_relerr_poly, col sep=comma]{IMG_apr26/error_mk_winner_T2_N2.csv};
\addplot[cNC, mark=triangle*, mark size=1.5pt, line width=1pt]
    table[x=logmoneyness, y=mean_relerr_poly, col sep=comma]{IMG_apr26/error_mk_winner_T2_N3.csv};
\addplot[cND, mark=diamond*,  mark size=1.5pt, line width=1pt]
    table[x=logmoneyness, y=mean_relerr_poly, col sep=comma]{IMG_apr26/error_mk_winner_T2_N4.csv};
\end{axis}
\end{tikzpicture}
&\hspace{0.35cm}
\begin{tikzpicture}
\begin{axis}[title={CME-LR, $T=1/2$}, xlabel={$\log(K/S_0)$},
    ymin=0, ymax=0.8, minor ytick={}, minor xtick={},
    width=6cm, height=5cm, grid=both,
    tick label style={font=\small}, label style={font=\small}, title style={font=\small\bfseries}]
\addplot[cNA, opacity=0.0, forget plot, name path=pathlo29]
    table[x=logmoneyness, y=lo_cme, col sep=comma]{IMG_apr26/error_mk_winner_T2_N1.csv};
\addplot[cNA, opacity=0.0, forget plot, name path=pathhi29]
    table[x=logmoneyness, y=hi_cme, col sep=comma]{IMG_apr26/error_mk_winner_T2_N1.csv};
\addplot[cNA, fill opacity=0.15, draw=none, forget plot]
    fill between[of=pathlo29 and pathhi29];
\addplot[cNB, opacity=0.0, forget plot, name path=pathlo30]
    table[x=logmoneyness, y=lo_cme, col sep=comma]{IMG_apr26/error_mk_winner_T2_N2.csv};
\addplot[cNB, opacity=0.0, forget plot, name path=pathhi30]
    table[x=logmoneyness, y=hi_cme, col sep=comma]{IMG_apr26/error_mk_winner_T2_N2.csv};
\addplot[cNB, fill opacity=0.15, draw=none, forget plot]
    fill between[of=pathlo30 and pathhi30];
\addplot[cNC, opacity=0.0, forget plot, name path=pathlo31]
    table[x=logmoneyness, y=lo_cme, col sep=comma]{IMG_apr26/error_mk_winner_T2_N3.csv};
\addplot[cNC, opacity=0.0, forget plot, name path=pathhi31]
    table[x=logmoneyness, y=hi_cme, col sep=comma]{IMG_apr26/error_mk_winner_T2_N3.csv};
\addplot[cNC, fill opacity=0.15, draw=none, forget plot]
    fill between[of=pathlo31 and pathhi31];
\addplot[cND, opacity=0.0, forget plot, name path=pathlo32]
    table[x=logmoneyness, y=lo_cme, col sep=comma]{IMG_apr26/error_mk_winner_T2_N4.csv};
\addplot[cND, opacity=0.0, forget plot, name path=pathhi32]
    table[x=logmoneyness, y=hi_cme, col sep=comma]{IMG_apr26/error_mk_winner_T2_N4.csv};
\addplot[cND, fill opacity=0.15, draw=none, forget plot]
    fill between[of=pathlo32 and pathhi32];
\addplot[cNA, mark=*,        mark size=1.5pt, line width=1pt, dashed]
    table[x=logmoneyness, y=mean_relerr_cme, col sep=comma]{IMG_apr26/error_mk_winner_T2_N1.csv};
\addplot[cNB, mark=square*,  mark size=1.5pt, line width=1pt, dashed]
    table[x=logmoneyness, y=mean_relerr_cme, col sep=comma]{IMG_apr26/error_mk_winner_T2_N2.csv};
\addplot[cNC, mark=triangle*, mark size=1.5pt, line width=1pt, dashed]
    table[x=logmoneyness, y=mean_relerr_cme, col sep=comma]{IMG_apr26/error_mk_winner_T2_N3.csv};
\addplot[cND, mark=diamond*,  mark size=1.5pt, line width=1pt, dashed]
    table[x=logmoneyness, y=mean_relerr_cme, col sep=comma]{IMG_apr26/error_mk_winner_T2_N4.csv};
\end{axis}
\end{tikzpicture}
\\[4pt]
\begin{tikzpicture}
\begin{axis}[title={LS, $T=1$}, xlabel={$\log(K/S_0)$}, ylabel={Mean relative error},
    ymin=0, ymax=0.8, minor ytick={}, minor xtick={},
    width=6cm, height=5cm, grid=both,
    tick label style={font=\small}, label style={font=\small}, title style={font=\small\bfseries}]
\addplot[cNA, opacity=0.0, forget plot, name path=pathlo33]
    table[x=logmoneyness, y=lo_poly, col sep=comma]{IMG_apr26/error_mk_winner_T3_N1.csv};
\addplot[cNA, opacity=0.0, forget plot, name path=pathhi33]
    table[x=logmoneyness, y=hi_poly, col sep=comma]{IMG_apr26/error_mk_winner_T3_N1.csv};
\addplot[cNA, fill opacity=0.15, draw=none, forget plot]
    fill between[of=pathlo33 and pathhi33];
\addplot[cNB, opacity=0.0, forget plot, name path=pathlo34]
    table[x=logmoneyness, y=lo_poly, col sep=comma]{IMG_apr26/error_mk_winner_T3_N2.csv};
\addplot[cNB, opacity=0.0, forget plot, name path=pathhi34]
    table[x=logmoneyness, y=hi_poly, col sep=comma]{IMG_apr26/error_mk_winner_T3_N2.csv};
\addplot[cNB, fill opacity=0.15, draw=none, forget plot]
    fill between[of=pathlo34 and pathhi34];
\addplot[cNC, opacity=0.0, forget plot, name path=pathlo35]
    table[x=logmoneyness, y=lo_poly, col sep=comma]{IMG_apr26/error_mk_winner_T3_N3.csv};
\addplot[cNC, opacity=0.0, forget plot, name path=pathhi35]
    table[x=logmoneyness, y=hi_poly, col sep=comma]{IMG_apr26/error_mk_winner_T3_N3.csv};
\addplot[cNC, fill opacity=0.15, draw=none, forget plot]
    fill between[of=pathlo35 and pathhi35];
\addplot[cND, opacity=0.0, forget plot, name path=pathlo36]
    table[x=logmoneyness, y=lo_poly, col sep=comma]{IMG_apr26/error_mk_winner_T3_N4.csv};
\addplot[cND, opacity=0.0, forget plot, name path=pathhi36]
    table[x=logmoneyness, y=hi_poly, col sep=comma]{IMG_apr26/error_mk_winner_T3_N4.csv};
\addplot[cND, fill opacity=0.15, draw=none, forget plot]
    fill between[of=pathlo36 and pathhi36];
\addplot[cNA, mark=*,        mark size=1.5pt, line width=1pt]
    table[x=logmoneyness, y=mean_relerr_poly, col sep=comma]{IMG_apr26/error_mk_winner_T3_N1.csv};
\addplot[cNB, mark=square*,  mark size=1.5pt, line width=1pt]
    table[x=logmoneyness, y=mean_relerr_poly, col sep=comma]{IMG_apr26/error_mk_winner_T3_N2.csv};
\addplot[cNC, mark=triangle*, mark size=1.5pt, line width=1pt]
    table[x=logmoneyness, y=mean_relerr_poly, col sep=comma]{IMG_apr26/error_mk_winner_T3_N3.csv};
\addplot[cND, mark=diamond*,  mark size=1.5pt, line width=1pt]
    table[x=logmoneyness, y=mean_relerr_poly, col sep=comma]{IMG_apr26/error_mk_winner_T3_N4.csv};
\end{axis}
\end{tikzpicture}
&\hspace{0.35cm}
\begin{tikzpicture}
\begin{axis}[title={CME-LR, $T=1$}, xlabel={$\log(K/S_0)$},
    ymin=0, ymax=0.8, minor ytick={}, minor xtick={},
    width=6cm, height=5cm, grid=both,
    tick label style={font=\small}, label style={font=\small}, title style={font=\small\bfseries}]
\addplot[cNA, opacity=0.0, forget plot, name path=pathlo37]
    table[x=logmoneyness, y=lo_cme, col sep=comma]{IMG_apr26/error_mk_winner_T3_N1.csv};
\addplot[cNA, opacity=0.0, forget plot, name path=pathhi37]
    table[x=logmoneyness, y=hi_cme, col sep=comma]{IMG_apr26/error_mk_winner_T3_N1.csv};
\addplot[cNA, fill opacity=0.15, draw=none, forget plot]
    fill between[of=pathlo37 and pathhi37];
\addplot[cNB, opacity=0.0, forget plot, name path=pathlo38]
    table[x=logmoneyness, y=lo_cme, col sep=comma]{IMG_apr26/error_mk_winner_T3_N2.csv};
\addplot[cNB, opacity=0.0, forget plot, name path=pathhi38]
    table[x=logmoneyness, y=hi_cme, col sep=comma]{IMG_apr26/error_mk_winner_T3_N2.csv};
\addplot[cNB, fill opacity=0.15, draw=none, forget plot]
    fill between[of=pathlo38 and pathhi38];
\addplot[cNC, opacity=0.0, forget plot, name path=pathlo39]
    table[x=logmoneyness, y=lo_cme, col sep=comma]{IMG_apr26/error_mk_winner_T3_N3.csv};
\addplot[cNC, opacity=0.0, forget plot, name path=pathhi39]
    table[x=logmoneyness, y=hi_cme, col sep=comma]{IMG_apr26/error_mk_winner_T3_N3.csv};
\addplot[cNC, fill opacity=0.15, draw=none, forget plot]
    fill between[of=pathlo39 and pathhi39];
\addplot[cND, opacity=0.0, forget plot, name path=pathlo40]
    table[x=logmoneyness, y=lo_cme, col sep=comma]{IMG_apr26/error_mk_winner_T3_N4.csv};
\addplot[cND, opacity=0.0, forget plot, name path=pathhi40]
    table[x=logmoneyness, y=hi_cme, col sep=comma]{IMG_apr26/error_mk_winner_T3_N4.csv};
\addplot[cND, fill opacity=0.15, draw=none, forget plot]
    fill between[of=pathlo40 and pathhi40];
\addplot[cNA, mark=*,        mark size=1.5pt, line width=1pt, dashed]
    table[x=logmoneyness, y=mean_relerr_cme, col sep=comma]{IMG_apr26/error_mk_winner_T3_N1.csv};
\addplot[cNB, mark=square*,  mark size=1.5pt, line width=1pt, dashed]
    table[x=logmoneyness, y=mean_relerr_cme, col sep=comma]{IMG_apr26/error_mk_winner_T3_N2.csv};
\addplot[cNC, mark=triangle*, mark size=1.5pt, line width=1pt, dashed]
    table[x=logmoneyness, y=mean_relerr_cme, col sep=comma]{IMG_apr26/error_mk_winner_T3_N3.csv};
\addplot[cND, mark=diamond*,  mark size=1.5pt, line width=1pt, dashed]
    table[x=logmoneyness, y=mean_relerr_cme, col sep=comma]{IMG_apr26/error_mk_winner_T3_N4.csv};
\end{axis}
\end{tikzpicture}
\\
\begin{tikzpicture}
\begin{axis}[title={LS, $T=2$}, xlabel={$\log(K/S_0)$}, ylabel={Mean relative error},
    ymin=0, ymax=0.8, minor ytick={}, minor xtick={},
    width=6cm, height=5cm, grid=both,
    tick label style={font=\small}, label style={font=\small}, title style={font=\small\bfseries}]
\addplot[cNA, opacity=0.0, forget plot, name path=pathlo41]
    table[x=logmoneyness, y=lo_poly, col sep=comma]{IMG_apr26/error_mk_winner_T4_N1.csv};
\addplot[cNA, opacity=0.0, forget plot, name path=pathhi41]
    table[x=logmoneyness, y=hi_poly, col sep=comma]{IMG_apr26/error_mk_winner_T4_N1.csv};
\addplot[cNA, fill opacity=0.15, draw=none, forget plot]
    fill between[of=pathlo41 and pathhi41];
\addplot[cNB, opacity=0.0, forget plot, name path=pathlo42]
    table[x=logmoneyness, y=lo_poly, col sep=comma]{IMG_apr26/error_mk_winner_T4_N2.csv};
\addplot[cNB, opacity=0.0, forget plot, name path=pathhi42]
    table[x=logmoneyness, y=hi_poly, col sep=comma]{IMG_apr26/error_mk_winner_T4_N2.csv};
\addplot[cNB, fill opacity=0.15, draw=none, forget plot]
    fill between[of=pathlo42 and pathhi42];
\addplot[cNC, opacity=0.0, forget plot, name path=pathlo43]
    table[x=logmoneyness, y=lo_poly, col sep=comma]{IMG_apr26/error_mk_winner_T4_N3.csv};
\addplot[cNC, opacity=0.0, forget plot, name path=pathhi43]
    table[x=logmoneyness, y=hi_poly, col sep=comma]{IMG_apr26/error_mk_winner_T4_N3.csv};
\addplot[cNC, fill opacity=0.15, draw=none, forget plot]
    fill between[of=pathlo43 and pathhi43];
\addplot[cND, opacity=0.0, forget plot, name path=pathlo44]
    table[x=logmoneyness, y=lo_poly, col sep=comma]{IMG_apr26/error_mk_winner_T4_N4.csv};
\addplot[cND, opacity=0.0, forget plot, name path=pathhi44]
    table[x=logmoneyness, y=hi_poly, col sep=comma]{IMG_apr26/error_mk_winner_T4_N4.csv};
\addplot[cND, fill opacity=0.15, draw=none, forget plot]
    fill between[of=pathlo44 and pathhi44];
\addplot[cNA, mark=*,        mark size=1.5pt, line width=1pt]
    table[x=logmoneyness, y=mean_relerr_poly, col sep=comma]{IMG_apr26/error_mk_winner_T4_N1.csv};
\addplot[cNB, mark=square*,  mark size=1.5pt, line width=1pt]
    table[x=logmoneyness, y=mean_relerr_poly, col sep=comma]{IMG_apr26/error_mk_winner_T4_N2.csv};
\addplot[cNC, mark=triangle*, mark size=1.5pt, line width=1pt]
    table[x=logmoneyness, y=mean_relerr_poly, col sep=comma]{IMG_apr26/error_mk_winner_T4_N3.csv};
\addplot[cND, mark=diamond*,  mark size=1.5pt, line width=1pt]
    table[x=logmoneyness, y=mean_relerr_poly, col sep=comma]{IMG_apr26/error_mk_winner_T4_N4.csv};
\end{axis}
\end{tikzpicture}
&\hspace{0.35cm}
\begin{tikzpicture}
\begin{axis}[title={CME-LR, $T=2$}, xlabel={$\log(K/S_0)$},
    ymin=0, ymax=0.8, minor ytick={}, minor xtick={},
    width=6cm, height=5cm, grid=both,
    tick label style={font=\small}, label style={font=\small}, title style={font=\small\bfseries}]
\addplot[cNA, opacity=0.0, forget plot, name path=pathlo45]
    table[x=logmoneyness, y=lo_cme, col sep=comma]{IMG_apr26/error_mk_winner_T4_N1.csv};
\addplot[cNA, opacity=0.0, forget plot, name path=pathhi45]
    table[x=logmoneyness, y=hi_cme, col sep=comma]{IMG_apr26/error_mk_winner_T4_N1.csv};
\addplot[cNA, fill opacity=0.15, draw=none, forget plot]
    fill between[of=pathlo45 and pathhi45];
\addplot[cNB, opacity=0.0, forget plot, name path=pathlo46]
    table[x=logmoneyness, y=lo_cme, col sep=comma]{IMG_apr26/error_mk_winner_T4_N2.csv};
\addplot[cNB, opacity=0.0, forget plot, name path=pathhi46]
    table[x=logmoneyness, y=hi_cme, col sep=comma]{IMG_apr26/error_mk_winner_T4_N2.csv};
\addplot[cNB, fill opacity=0.15, draw=none, forget plot]
    fill between[of=pathlo46 and pathhi46];
\addplot[cNC, opacity=0.0, forget plot, name path=pathlo47]
    table[x=logmoneyness, y=lo_cme, col sep=comma]{IMG_apr26/error_mk_winner_T4_N3.csv};
\addplot[cNC, opacity=0.0, forget plot, name path=pathhi47]
    table[x=logmoneyness, y=hi_cme, col sep=comma]{IMG_apr26/error_mk_winner_T4_N3.csv};
\addplot[cNC, fill opacity=0.15, draw=none, forget plot]
    fill between[of=pathlo47 and pathhi47];
\addplot[cND, opacity=0.0, forget plot, name path=pathlo48]
    table[x=logmoneyness, y=lo_cme, col sep=comma]{IMG_apr26/error_mk_winner_T4_N4.csv};
\addplot[cND, opacity=0.0, forget plot, name path=pathhi48]
    table[x=logmoneyness, y=hi_cme, col sep=comma]{IMG_apr26/error_mk_winner_T4_N4.csv};
\addplot[cND, fill opacity=0.15, draw=none, forget plot]
    fill between[of=pathlo48 and pathhi48];
\addplot[cNA, mark=*,        mark size=1.5pt, line width=1pt, dashed]
    table[x=logmoneyness, y=mean_relerr_cme, col sep=comma]{IMG_apr26/error_mk_winner_T4_N1.csv};
\addplot[cNB, mark=square*,  mark size=1.5pt, line width=1pt, dashed]
    table[x=logmoneyness, y=mean_relerr_cme, col sep=comma]{IMG_apr26/error_mk_winner_T4_N2.csv};
\addplot[cNC, mark=triangle*, mark size=1.5pt, line width=1pt, dashed]
    table[x=logmoneyness, y=mean_relerr_cme, col sep=comma]{IMG_apr26/error_mk_winner_T4_N3.csv};
\addplot[cND, mark=diamond*,  mark size=1.5pt, line width=1pt, dashed]
    table[x=logmoneyness, y=mean_relerr_cme, col sep=comma]{IMG_apr26/error_mk_winner_T4_N4.csv};
\end{axis}
\end{tikzpicture}
\\
\multicolumn{2}{c}{%
\begin{tikzpicture}
\draw[cNA, line width=1pt] (0,0) -- (0.5,0);
\fill[cNA] (0.25,0) circle (1.5pt);
\node[right] at (0.5,0) {\small $n=10^2$};
\draw[cNB, line width=1pt] (2.8,0) -- (3.3,0);
\fill[cNB] (3.05,0) circle (1.5pt) node[draw=none,inner sep=0]{};
\node[right] at (3.3,0) {\small $n=10^3$};
\draw[cNC, line width=1pt] (5.6,0) -- (6.1,0);
\fill[cNC] (5.85,0) circle (1.5pt);
\node[right] at (6.1,0) {\small $n=10^4$};
\draw[cND, line width=1pt] (8.4,0) -- (8.9,0);
\fill[cND] (8.65,0) circle (1.5pt);
\node[right] at (8.9,0) {\small $n=10^5$};
\end{tikzpicture}}
\end{tabular}
\caption{Mean relative implied volatility error vs.\ $\log(K/S_0)$. 
  Left: LS (solid). Right: CME-LR offline (dashed).
Shaded bands: 95\% confidence intervals over 100 replications.
}
\label{fig:moneyness}
\end{figure}
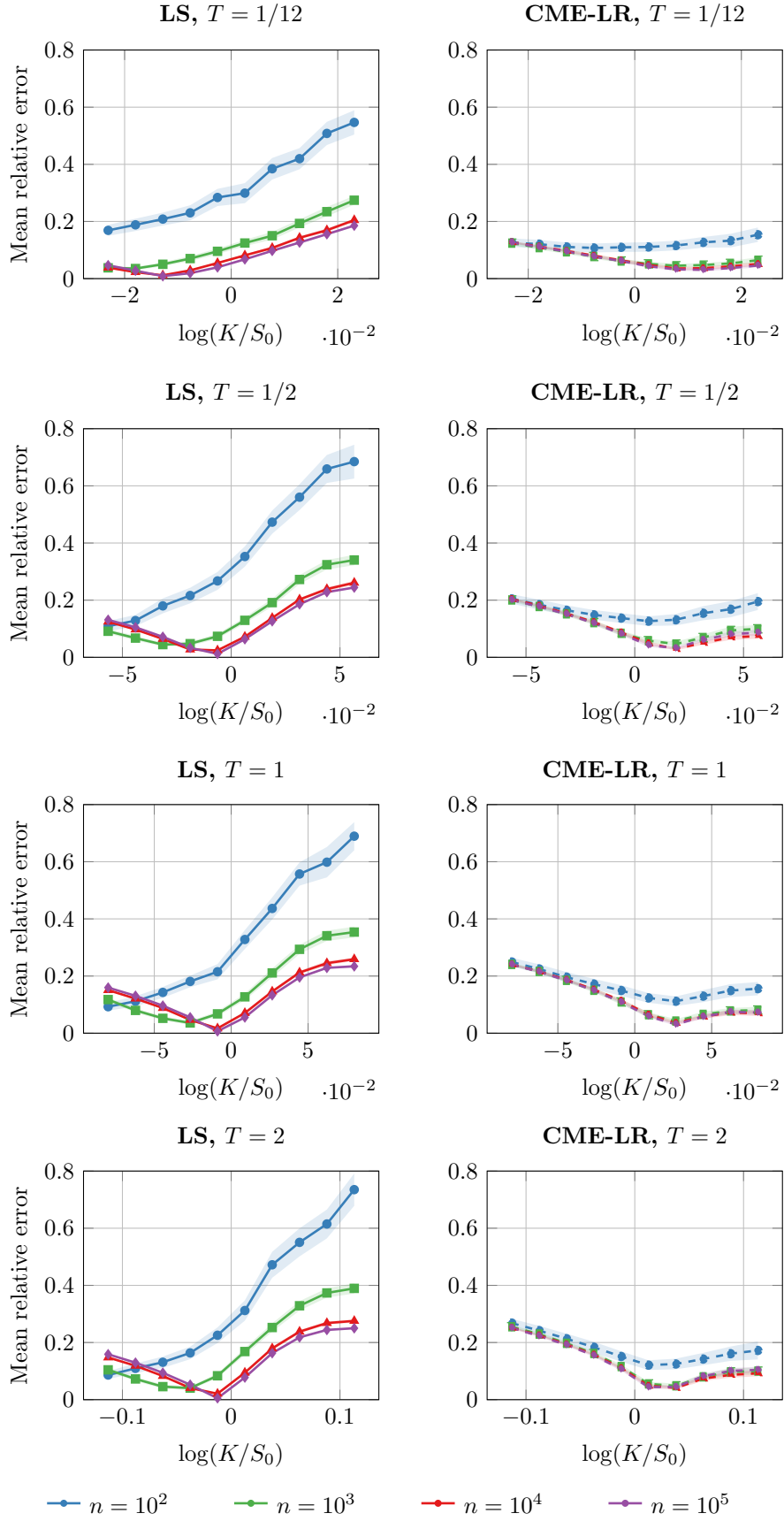
For the  simulation grid, we consider all combinations of the following parameters:
number of simulation paths $n \in \{10^2, 10^3, 10^4, 10^5\}$,
maturities
$T \in \{1/12,\, 1/2,\, 1,\, 2\}$ years,
and strikes $K = S_0 \exp(m \sqrt{\nu_0} \sqrt{T})$
with $m \in \operatorname{linspace}(-2, 2, 10)$.
Each configuration is repeated over $100$ independent replications,
with seeds $\texttt{seed} = \texttt{rep} \times 16 + n_i \times 4 + t_i$
to ensure reproducibility.
{The benchmark American put price is computed once via the \textsc{Matlab}
finite-difference solver \texttt{optByHestonFD}, while the CME-LR and LS
pricing algorithms are implemented in \textsc{C++}. Computation times are
measured around the full pricing algorithm, excluding path generation and data
preparation.}

Performance is measured via the mean relative implied volatility error,
averaged over the $10$ strikes and the $100$ replications:
\begin{equation}\label{eq:err_IV}
\overline{\varepsilon}_{\mathrm{rel}}(n, T)
\;=\;
\frac{1}{100} \sum_{r=1}^{100}
\frac{1}{10} \sum_{k=1}^{10}
\frac{\bigl|\widetilde{\mathrm{IV}}^{(r)}_k(n,T) 
- \mathrm{IV}^{\mathrm{ref}}_k(T)\bigr|}
     {\mathrm{IV}^{\mathrm{ref}}_k(T)},
\end{equation}
where $\widetilde{\mathrm{IV}}^{(r)}_k$ denotes the implied volatility
from the simulated price in replication $r$ and strike $k$,
and $\mathrm{IV}^{\mathrm{ref}}_k$ is the reference implied volatility
from the \texttt{optByHestonFD} benchmark.
Here, the implied volatility $\widetilde{\mathrm{IV}}^{(r)}_k$ is obtained by 
numerically inverting the Black-Scholes European put pricing formula applied 
to the simulated American put price, and analogously 
$\mathrm{IV}^{\mathrm{ref}}_k$ is obtained from the benchmark price. 

Figure~\ref{fig:timing} reports the mean $\log_{10}$ computation time as a
function of $n$ for both methods across the four maturities,
over the $100$ replications.
CME-LR is systematically faster than LS for all values of $n$ and $T$,
with the gap being most pronounced for small $n$. This speed-up is likely
induced from the low-rank basis being particularly parsimonious for the model
in question, and the computational savings from not repeatedly recomputing the
coefficients of the conventional LS algorithm.

To illustrate how parsimonious the low-rank basis is, Figure~\ref{fig:rank}
displays the mean rank of $\bm{K}_Y$ and $\bm{K}_X$ selected 
by the pivoted Cholesky decomposition as a function of $n$, for three values of
the tolerance $\varepsilon \in \{10^{-4}, 10^{-5}, 10^{-6}\}$. While all pricing
results reported in this section are obtained with $\varepsilon = 10^{-5}$, we
include the other two tolerances for comparison. As expected, a smaller
$\varepsilon$ leads to a higher rank, reflecting a more accurate approximation
of the kernel matrix. We do not separately plot the rank of $\bm{K}_X$, as it
remains small across all configurations and tolerances, taking values in
$\{2,3\}$ for $\varepsilon \in \{10^{-4}, 10^{-5}\}$ and between $3$ and $6$ for
$\varepsilon = 10^{-6}$. This pronounced rank deficiency is consistent with the
strongly negative correlation $\rho = -0.7$ used in our Heston configuration,
a feature reminiscent of equity index markets such as the S\&P~500, in which the
leverage effect manifests as a marked implied volatility skew. Such a high
negative correlation effectively concentrates the joint distribution of
$(\log S_k, \nu_k)$ near a one-dimensional affine subspace, which translates into
a low-dimensional linear manifold within the polynomial feature space and is
captured by the pivoted Cholesky decomposition through the low-rank
representation of $\bm{K}_X$.
In contrast, the rank of $\bm{K}_Y$ grows with $n$ for all tolerances, confirming
that the low-rank approximation adapts to the complexity of the transition kernel
as more data become available.

Figure~\ref{fig:error_agg} shows the mean relative implied volatility error 
$\overline{\varepsilon}_{\mathrm{rel}}(n, T)$ in Equation~\eqref{eq:err_IV} 
averaged over strikes as a function of $n$.
CME-LR achieves uniformly lower error than LS across all maturities.

Finally, Figure~\ref{fig:moneyness} shows the per-strike version of the same 
error \eqref{eq:err_IV}, averaged over replications only. In particular, 
the error is displayed as a function of log-moneyness
$\log(K/S_0)$ for each maturity and value of $n$.
CME-LR achieves lower error in the out-of-the-money region
(positive log-moneyness), while LS performs better for in-the-money options
(negative log-moneyness).

\section{Conclusion}\label{sec:conclusion}
We have developed a kernel-based framework for American option pricing that
exploits the offline-online decomposition enabled by the conditional mean
embedding approach. By treating continuation value estimation as an operator
learning problem, the method avoids refitting a regression model at each
exercise date, leading to systematic computational gains over the
Longstaff-Schwartz baseline.

The low-rank approximation via pivoted Cholesky decomposition keeps the
computational cost tractable while maintaining controlled approximation error,
as established by the theoretical bounds derived in
Sections~\ref{sec:conv_rate} and~\ref{sec:offline_cme}.
In particular, the total pricing error grows at most linearly in the number
of exercise dates, with a proportionality constant that decreases as the
sample size and Cholesky tolerance improve.

Numerical experiments on the Heston model confirm that CME-LR achieves
uniformly lower mean relative implied volatility error than Longstaff-Schwartz
across all maturities and sample sizes, with a complementary moneyness
profile: CME-LR is more accurate in the out-of-the-money region, while
Longstaff-Schwartz retains an advantage for in-the-money options. The rank
of the kernel matrix $\bm{K}_Y$ grows sub-linearly with $n$, confirming
that the low-rank structure adapts automatically to the data complexity.

\section*{Acknowledgments}
The authors thank the Swiss National Science Foundation (SNSF) for the financial
support through the grant number 215528, ``Large-scale kernel methods in
financial economics". C. Segala is a member of the Italian National Group of
Scientific Calculus (INdAM-GNCS, participant of the project E53C25002010001).

\appendix
\section{Proof of Theorem~\ref{thm:approx}}\label{app:proof_lowrank}
\begin{proof}
	The expression on the right-hand side of \eqref{eq:orthoproj} follows from
  applying matrix calculus. Furthermore, by orthogonality of the decomposition
  $\hat \mu - \hat \mu_{\tilde\Hcal}$ and $\hat \mu_{\tilde\Hcal} - \tilde \mu$,
  we have
	\begin{align*}
		\|\hat \mu -\tilde \mu\|_{\Hcal}^2 =\| 
    \hat \mu  - \hat \mu_{\tilde\Hcal}\|_{\Hcal}^2 
    + \| \hat \mu_{\tilde\Hcal}-\tilde \mu\|_{\Hcal}^2 =\| \hat \mu  
    - \hat \mu_{\tilde\Hcal}\|_{\Hcal}^2 
    + \big\| \bm F_{\tilde\Hcal}-\tilde{\bm F} \big\|_F^2.
	\end{align*}
	It remains to bound the first term on the right-hand side
	\begin{align*}
		\|\hat \mu- \hat \mu_{\tilde\Hcal}\|_\Hcal^2&=
		\big\|{\bs\Phi}_Y(\cdot)\bm F {\bs\Phi}_X(\cdot)^\T-
    \bs\Phi_Y (\cdot) \bm B_Y \bm V_{Y}  \ \bm F_{\tilde\Hcal} \
		\left(\bm B_X \bm V_{X} \right)^\T \bs\Phi _X(\cdot)^{\T}  \big\|_\Hcal^2\\
		&=\big\|{\bs\Phi}_Y(\cdot)(\bm F -\bm B_Y\bm L_{Y}^{\T}
    \bm F\bm L_{X}\bm B_X^{\T}) {\bs\Phi}_X(\cdot)^\T\big\|_\Hcal^2\\
		&=\trace \big ((\bm F -\bm B_Y\bm L_{Y}^{\T}\bm F\bm L_{X}
    \bm B_X^{\T})^{\T}\bm K_Y (\bm F -\bm B_Y\bm L_{Y}^{\T}
  \bm F\bm L_{X}\bm B_X^{\T})\bm K_X \big ) \\
		&=\trace \big (\bm F^{\T}\bm K_Y \bm F \bm K_X \big ) 
    -2\trace \big (\bm F ^{\T}\bm K_Y \bm B_Y\bm L_{Y}^{\T}
    \bm F\bm L_{X}\bm B_X^{\T} \bm K_X\big) \\
		&\qquad+\trace \big (\bm B_X\bm L_{X}^{\T}
    \bm F^{\T}\bm L_{Y}\bm B_Y^{\T} \bm K_Y \bm B_Y
    \bm L_{Y}^{\T}\bm F\bm L_{X}\bm B_X^{\T} \bm K_X\big)\\
		&=\trace \big (\bm F^{\T}\bm K_Y \bm F \bm K_X \big )
    -\trace \big (\bm F^{\T}\bm L_Y\bm L_Y^{\T} 
    \bm F \bm L_X \bm L_X^{\T} \big )\\
		&=\trace \big (\bm F^{\T}\bm K_Y \bm F (\bm K_X-\bm L_X\bm L_X^\T) \big)
		+\trace \big (\bm F^{\T}(\bm K_Y-\bm L_Y\bm L_Y^{\T})\bm F 
    \bm L_X \bm L_X^{\T} \big )\\
		&=\langle\bm K_Y, \bm F (\bm K_X-\bm L_X\bm L_X^\T)\bm F^\T\rangle_F
		+\langle\bm L_X \bm L_X^{\T},
    \bm F^\T(\bm K_Y-\bm L_Y\bm L_Y^{\T})\bm F\rangle_F
		\\
		&\leq\|{\bs F}\|_F^2\|{\bs K}_Y\|_F
		\|\bm K_X-\bm L_X\bm L_X^\T \|_F+\|{\bs F}\|_F^2\|\bm L_X \bm L_X^{\T}\|_F
		\|\bm K_Y-\bm L_Y\bm L_Y^\T \|_F\\
		&\leq \varepsilon \|{\bs F}\|_F^2\big(\trace({\bs K}_X)+\trace({\bs K}_Y)),
	\end{align*}
	where we used the Cauchy-Schwarz inequality for the Frobenius inner product, 
  the sub-multiplicativity of the Frobenius norm, the bound 
  \(\|{\bs A}\|_F\leq\trace{\bs A}\) for any positive semidefinite matrix 
  $\bs A$, and the inequality 
  \(\trace\bm L_X \bm L_X^{\T}\leq\trace({\bs K}_X)\).
\end{proof}

\section{Proof of Proposition \ref{prop:coeff_err}}\label{app:proof_coeff}
\begin{proof}
     As a first step of the proof, we aim to show that the two matrices
	\[
	\tilde{\bm F}
	=
	(\bm L_Y \bm V_Y)^\top (\bm L_X \bm V_X) 
  (\bm \Lambda_X + n \lambda \bm I_{m_X})^{-1},
	\quad
	\bar{\bm F}
	=
	(\bm L_Y \bm V_Y)^\top 
  (\bm L_X \bm L_X^\top + n \lambda \bm I_n)^{-1} \bm L_X \bm V_X
	\]
	are identical. 
The matrix $\tilde{\bm F}$ is the coefficient matrix 
    of the low-rank approximation $\tilde{\mu} \in \tilde{\Hcal}$,
while $\bar{\bm F}$ arises naturally when replacing the full kernel
matrix $\bm K_X$ with its low-rank 
    approximation $\bm L_X \bm L_X^\top$ in the regularized inverse. 
	We first recall the Woodbury matrix identity
	\[
	(\bm L_X \bm L_X^\top + n \lambda \bm I_n)^{-1} \bm L_X
	= \tfrac{1}{n \lambda}\!\left( \bm I_n
	- \bm L_X (\bm L_X^\top \bm L_X + n 
\lambda \bm I_{m_X})^{-1} \bm L_X^\top \right) \bm L_X.
	\]
	Multiplying both sides by $\bm V_X$ and using 
  $\bm L_X^\top \bm L_X = \bm V_X \bm \Lambda_X \bm V_X^\top$, we obtain
	\begin{align*}
		(\bm L_X \bm L_X^\top + n \lambda \bm I_n)^{-1} \bm L_X \bm V_X
		&= \frac{1}{n \lambda}\!\left[
		\bm L_X \bm V_X
		- \bm L_X \bm V_X
		(\bm \Lambda_X + n \lambda \bm I_{m_X})^{-1} \bm \Lambda_X
		\right] \\
		&= \bm L_X \bm V_X (\bm \Lambda_X + n \lambda \bm I_{m_X})^{-1}.
	\end{align*}
	The last equality follows from the scalar identity, valid elementwise 
  for each eigenvalue $\lambda_i$:
	\[
	\frac{1}{n \lambda}\!\left(1 - \frac{\lambda_i}{\lambda_i + n \lambda}\right)
	= \frac{1}{\lambda_i + n \lambda}.
	\]
	Substituting back, we find
	\begin{align*}
    \bar{\bm F}
	= (\bm L_Y \bm V_Y)^\top (\bm L_X \bm V_X)(\bm \Lambda_X + n \lambda \bm I_{m_X})^{-1}
    = \tilde{\bm F}.
	\end{align*}
Note that although Proposition~\ref{prop:coeff_err} is only stated for $\lambda > 0$,
the identity $\tilde{\bm F} = \bar{\bm F}$ holds for $\lambda = 0$ as well.

We are now ready to compute $
	{\| \bm F_{\tilde\Hcal}-\tilde{\bm F} \|_F^2}=
	{\| \bm F_{\tilde\Hcal}-\bar{\bm F} \|_F^2}$.
	 Define
	\[
	\bm A \isdef \bm F - (\bm L_X\bm L_X^\top + n\lambda \bm I_n)^{-1},
	\]
	so that
	\[
	\bm F_{\tilde\Hcal} - \bar{\bm F}
	= (\bm L_Y\bm V_Y)^\top  \bm A  \bm L_X\bm V_X.
	\]
	Using the submultiplicative property of the Frobenius norm, one has
	\begin{equation}\label{eq:three_prod}
		\| (\bm L_Y\bm V_Y)^\top \bm A \bm L_X\bm V_X \|_F
		\le
		\| \bm L_Y\bm V_Y \|_F \| \bm A \|_F  \| \bm L_X\bm V_X \|_F.
	\end{equation}	
	Set \(\bm M \isdef \bm L_X\bm L_X^\top + n\lambda \bm I_n\) 
  and \(\bm R \isdef \bm K_X - \bm L_X\bm L_X^\top\),
  so that \(\trace(\bm R)\le\varepsilon\) and \(\bm R\) is positive 
  semidefinite, from Theorem \ref{thmChol}). Then \(\bm F = (\bm M + \bm R)^{-1}\) 
  and the resolvent identity gives
	\[
	\bm A = (\bm M + \bm R)^{-1} 
  - \bm M^{-1} = -\bm M^{-1} \bm R (\bm M + \bm R)^{-1}.
	\]
    Taking Frobenius norm and using submultiplicativity, we first write
\[
\| \bm A \|_F \le \| \bm M^{-1} \|_2 \, \| \bm R \|_F \, \| (\bm M + \bm R)^{-1} \|_2,
\]
where $\| \cdot \|_2$ is the spectral norm. In the following, we denote by $\succeq$ 
the positive semidefinite order for symmetric matrices. 
Since $\bm M \succeq n\lambda\bm I_n,$, because $\bm M - n\lambda \bm I_n$
is positive semidefinite, there holds
    \[
    \| \bm M^{-1} \|_2 \le \frac{1}{n\lambda}, \quad \text{and likewise } \| 
    (\bm M + \bm R)^{-1} \|_2 \le \frac{1}{n\lambda}.
    \]
    Therefore
	\begin{equation*}
		\| \bm A \|_F \le \frac{1}{(n\lambda)^2} \| \bm R \|_F.
	\end{equation*}
	Finally, because \(\bm R\) is symmetric positive semidefinite, its Frobenius 
  norm satisfies \(\| \bm R \|_F = \sqrt{\sum_i \lambda_i(\bm R)^2} 
  \le \sum_i \lambda_i(\bm R) = \operatorname{trace}(\bm R)\), where 
  $\lambda_i(\bm R)$ denotes the $i$-th eigenvalue of $\bm R$. Using 
  \(\operatorname{trace}(\bm R)\le\varepsilon\) yields
	\[
	\| \bm A \|_F \le \frac{\varepsilon}{(n\lambda)^2}.
	\]
	Inserting this bound into \eqref{eq:three_prod} gives
	\[
	\| \bm F_{\tilde\Hcal} - \bar{\bm F} \|_F
	\le
	\frac{\varepsilon}{(n\lambda)^2}
  \| \bm L_X \bm V_X \|_F \| \bm L_Y \bm V_Y \|_F.
	\]
 Squaring both sides and applying 
 \(\| \bm L_X \bm V_X \|_F \le \sqrt{\operatorname{trace}(\bm K_X)}\) 
 and \(\| \bm L_Y \bm V_Y \|_F \le \sqrt{\operatorname{trace}(\bm K_Y)}\) 
 yields the squared bound
 \[
\| \bm F_{\tilde\Hcal} - \tilde{\bm F} \|_F^2
\leq
\frac{\varepsilon^2}{(n\lambda)^4}
\operatorname{trace}(\bm K_X)\,\operatorname{trace}(\bm K_Y).
\]

\end{proof}

\section{CME Statistical error}\label{app:stat_error}
The study of convergence rates for regularized estimators in RKHS has a rich
history in the kernel methods literature. Early consistency results for
conditional mean embeddings were established by \citet{songetal09}, while
\citet{gruenewaelderetal12} derived minimax optimal 
rates in the well-specified setting under the restrictive assumption that 
$\mathcal{H}_\mathcal{Y}$ is finite-dimensional. For infinite-dimensional output
spaces, \citet{Caponnetto2007} provide the first convergence analysis of 
regularized least-squares with vector-valued outputs, but their results rely on 
a trace class condition on the kernel operator which is violated for the
standard choice used in CME estimation, 
and they only cover the well-specified case. This limitation was noted in the 
CME context by \citet{gruenewaelderetal12} and addressed in subsequent work. 
\citet{NEURIPS2020_f340f1b1} and \citet{HouBoya_etal2023} establish consistency 
results in the well-specified setting without the trace class restriction, 
while \citet{li2022optimal} derive sharp convergence rates for CME learning 
covering both the well-specified and misspecified settings for 
infinite-dimensional $\mathcal{H}_\mathcal{Y}$, under a boundedness assumption 
on the regression function. This boundedness requirement is subsequently removed 
in \citet{LiMeunier_etal2024}, which provides the most general convergence 
theory for vector-valued ridge regression, of which the CME is a special case, 
to date, and whose results we use here to bound the statistical error term 
$\|\mu - \hat{\mu}\|_{\gamma}$ in \eqref{eq:stat_err_Li}.

In the notation of \citet{LiMeunier_etal2024}, we denote by $\pi$ the marginal 
distribution of $X$ under $\mathbb{Q}$, by $\nu$ the marginal distribution 
of $Y$, and by $p: \Xcal \times \mathcal{F} \to \mathbb{R}_+$ 
the Markov kernel characterizing the conditional distribution of $Y$ given $X$, 
i.e., 
$$
\mathbb{Q}[Y \in A \mid X = x] = \int_A p(x, dy),\quad\text{for all x }
\in \Xcal
$$
and events $A \in \Fcal$, see, e.g., \citet{dud_02}.
We now state the assumptions required by their result, translated into our 
notation.

\begin{itemize}[itemsep=0.4em]
    \item[\textbf{(H1)}] $\Hcal_\Xcal$ is separable.

    \item[\textbf{(H2)}] $k_\Xcal(\cdot, x)$ is measurable for $\pi$-almost 
      all $x \in \Xcal$.

    \item[\textbf{(H3)}] There exists a constant $\kappa_X > 0$ such that
    $k_\Xcal(x,x) \leq \kappa_X^2$ for $\pi$-almost all $x \in \Xcal$.

  \item[\textbf{(H4)}] The eigenvalues $\{\eta_i\}_{i \in I}$ of the integral
    operator associated with $k_\Xcal$ satisfy, for some constants
    $c_2 > 0$ and $p \in (0,1]$,
    \[
      \eta_i \leq c_2\, i^{-1/p} \quad\text{for all } i \in I.
    \]
    \item[\textbf{(H5)}] For $\alpha \in [p, 1]$, the inclusion map
    $I^{\alpha,\infty}_\pi : [\Hcal_\Xcal]^\alpha \hookrightarrow
    L_\infty(\pi)$ is continuous with constant $A > 0$:
    \[
    \|I^{\alpha,\infty}_\pi\|_{[\Hcal_\Xcal]^\alpha
    \to L_\infty(\pi)} \leq A.
    \]

    \item[\textbf{(H6)}] There exist $0 < \beta \leq 2$ and a constant
    $B \geq 0$ such that the true CME $\mu_{Y|X=\cdot}$ belongs to the
    interpolation space $[\Hcal]^\beta$ with $\|\mu\|_\beta \leq B$.

   \item[\textbf{(H7)}] There exist constants $\sigma, R > 0$ such that
    \[
    \int_{\Hcal_\Ycal} \|y - \mu(x)\|^q_{\Hcal_\Ycal}
    \, p(x,dy) \leq \tfrac{1}{2}\, q!\, \sigma^2 R^{q-2}
    \]
    for $\pi$-almost all $x \in \Xcal$ and all integers $q \geq 2$.
\end{itemize}
The following result is Theorem~3 of \citet{LiMeunier_etal2024}, 
restated in our notation.

\begin{theorem}[\citet{LiMeunier_etal2024}, Theorem~3]\label{thm:stat_error}
Let \textnormal{(H1)--(H7)} hold with $0 < \beta \leq 2$, 
and let $0 \leq \gamma \leq 1$ with $\gamma < \beta$.
\begin{enumerate}
    \item If $\beta + p \leq \alpha$ and 
    $\lambda_n = \Theta\!\left(\left(n/\log^\theta n\right)^{-1/\alpha}\right)$ 
    for some $\theta > 1$, then for all $\tau > \log(5)$ and 
    sufficiently large $n$, with probability at least $1 - 5e^{-\tau}$:
    \[
    \|\hat{\mu} - \mu\|^2_{\gamma} 
    \leq \tau^2 c_1 \left(\frac{n}{
    \log^\theta n}\right)^{-\frac{\beta-\gamma}{\alpha}}.
    \]
    \item If $\beta + p > \alpha$ and 
    $\lambda_n = \Theta\!\left(n^{-1/(\beta+p)}\right)$, 
    then for all $\tau > \log(5)$ and sufficiently large $n$, 
    with probability at least $1 - 5e^{-\tau}$:
    \[
    \|\hat{\mu} - \mu\|^2_{\gamma} 
    \leq \tau^2 c_1\, n^{-\frac{\beta-\gamma}{\beta+p}}.
    \]
\end{enumerate}
Here $c_1 > 0$ is a constant independent of $n$ and $\tau$, 
and $\tau > \log(5)$ is a confidence parameter.
\end{theorem}

\bibliographystyle{plainnat}
\bibliography{mss26_BIB}

\end{document}